\documentclass{article}[11pt]
\usepackage{latexsym}
\usepackage{amsmath}
\usepackage{amsfonts}
\usepackage{amssymb}
\usepackage{mathrsfs}
\usepackage{graphicx}	
\usepackage{amsthm}
\usepackage[top=2.5cm, bottom=2.5cm, left=2.5cm, right=2.5cm]{geometry}
\usepackage{hyperref}
\usepackage{slashed,mathtools,color}
\usepackage{upgreek}
\usepackage{bbm}
\usepackage{comment}
\usepackage{xcolor}
\usepackage{authblk}

\newcommand{\beq}{\begin{equation}}
\newcommand{\eeq}{\end{equation}}


\theoremstyle{plain}
\newtheorem{theorem}{Theorem}[section]
\newtheorem{proposition}[theorem]{Proposition}
\newtheorem{lemma}[theorem]{Lemma}
\newtheorem{corollary}[theorem]{Corollary}

\theoremstyle{definition}

\newtheorem{remark}[theorem]{Remark}

\numberwithin{equation}{section}

\allowdisplaybreaks[4] 

\swapnumbers
\begin{document}

\title{Future dynamics of FLRW for the massless-scalar field system with positive cosmological constant}

\author{Grigorios Fournodavlos\footnote{Princeton University, Mathematics Department, Fine Hall,
 Washington Road, Princeton, NJ 08544-1000, USA. email: gf5797@princeton.edu}}
%\address{}

\date{}

\maketitle

\begin{abstract}
We consider solutions to the Einstein-massless-scalar field system with a positive cosmological constant, arising from sufficiently regular, near-FLRW, initial data. We establish global existence in the future direction and derive their precise asymptotic behavior towards infinity. 
As a corollary, we infer that, unlike the FLRW background, the perturbed solutions do not describe a regular irrotational stiff fluid with linear equation of state $p=\rho$, for general asymptotic data at infinity. The reason for the breakdown of this interpretation is that the gradient of the scalar field stops being timelike at large times, eventually becoming null and then spacelike near infinity. Our results hold for open sets of initial data in Sobolev spaces without symmetries.
\end{abstract}

\parskip = 0 pt

%\tableofcontents
%
\section{Introduction}

The present paper is concerned with the large time behavior of solutions to the Einstein-massless-scalar field system with a positive cosmological constant:
\begin{align}
\label{EE}
{\bf Ric}_{\mu\nu}=&\,\Lambda{\bf g}_{\mu\nu}+\partial_\mu\psi\partial_\nu\psi,\\
\label{divT}\square_{\bf g}\psi=&\,0,
\end{align}
whose initial data ($\Sigma_0=\mathbb{T}^3,\mathring{g},\mathring{k},\mathring{\psi},\mathring{\varphi}$) are perturbations of the initial data induced by the FLRW solution \eqref{FLRW.sol} on $\{t=0\}$, without symmetries and of Sobolev regularity, satisfying the constraint equations: 
\begin{align}
\label{Hamconst.intro}\mathring{R}-|\mathring{k}|^2+(\text{tr}\mathring{k})^2=&\,2\Lambda+\mathring{\varphi}^2+|\mathring{\nabla}\mathring{\psi}|^2,\\
\label{momconst.intro}\mathring{\text{div}}\mathring{k}-\text{d}\text{tr}\mathring{k}=&-\mathring{\varphi}\mathring{\nabla}\mathring{\psi},
\end{align}
where $\mathring{g}$ is a Riemannian metric, $\mathring{\nabla}$ is the Levi-Civita connection of $\mathring{g}$ (with respect to which all covariant spatial operations are defined on $\Sigma_0$),  $\mathring{R}$ is the scalar curvature of $\mathring{g}$, $\mathring{k}$ is a symmetric 2-tensor, and $\mathring{\psi},\mathring{\varphi}$ are functions on $\Sigma_0$. The classical theorem of Choquet-Bruhat--Geroch \cite{BG} implies that for each sufficiently regular initial data set, satisfying \eqref{Hamconst.intro}-\eqref{momconst.intro}, there exists a corresponding unique (up to isometry) maximal development $({\bf g},\mathcal{M})$, along with an embedding $i:\Sigma_0\to\mathcal{M}$ such that $i(\Sigma_0)$ is a Cauchy hypersurface, the pull backs of its first and second fundamental forms are equal to $\mathring{g},\mathring{k}$, and the pull backs of the initial values of the scalar field and its derivative with respect to the future unit normal to $i(\Sigma_0)$ are equal to $\mathring{\psi},\mathring{\phi}$ respectively.\footnote{For simplicity, we abuse notation by not referring to this identification of the initial data through the embedding in the rest of the paper.} In this paper, $i(\Sigma_0)$ coincides with the level set $\{t=0\}$ of an appropriately normalized time function $t$.

Our main results show that the spacetime metrics exist globally in the future direction, ie. over $[0,+\infty)\times\mathbb{T}^3$, where the time parameter coincides to leading order with proper time near future null infinity, which in this context is spacelike, and the spatial part of the metrics converges exponentially to a near-isotropic end state (see Theorem \ref{thm:main}). Moreover, the scalar field converges to a limit function with its time derivative tending exponentially to zero. One could interpret these results as future stability of FLRW, albeit the time derivative of the scalar field decaying to zero at a different rate than that of the background, cf. \eqref{thm.main.est2}, \eqref{red.var.FLRW}. Interestingly, the perturbed solutions break down, in general, when viewed as irrotational stiff fluids with linear equation of state $p=\rho$, due to the gradient of the scalar field becoming spacelike near infinity (see Corollary \ref{cor:fluid}). 
%Such a behavior, due to large spatial gradients, was already suggested by Rendall in \cite{Ren}, where he derived formal expansions of solutions to the Einstein-Euler system with a positive cosmological constant and linear equation of state $p=c_s^2\rho$, for $0\leq c_s^2<\frac{1}{3}$. 
We elaborate more on this phenomenon in Section \ref{subsec:thm} (see the discussion before Corollary \ref{cor:fluid}).

There has been a lot of work on the global dynamics of cosmological solutions with an accelerated expansion. In fact, the first global understanding of dynamical solutions to the Einstein equations \eqref{EE}, without symmetries, is due to Friedrich \cite{Fr} in the vacuum case (ie. $\psi=0$, $\Lambda>0$). Friedrich's classical result has been subsequently generalized to higher dimensions \cite{And,FK} and for spatial sections of negative scalar curvature \cite{MV}. An important breakthrough was achieved by Ringstr\"om \cite{Rin}, for the Einstein-scalar field system with a potential, where he studied the global future behavior of solutions with close to trivial initial data and general (compact) spatial topologies. In his work, the specific potential is driving the accelerated expansion, without having to add a cosmological constant to the system. Although the model studied in the present paper is not covered by the results in \cite{Rin}, since the effect of the cosmological constant is only implemented by the scalar field potential, the methods developed therein could be adapted to also treat the massless-scalar field with $\Lambda>0$, but with less than optimal estimates for the scalar field.\footnote{Compare (29) in \cite[Theorem 2]{Rin} with \eqref{thm.main.est2}-\eqref{thm.main.est3} in Theorem \ref{thm:main} below. The sharp decay for the variables $\psi-\psi^\infty,e_0\psi$ in our model is $e^{-2Ht}$, whereas \cite[Theorem 2]{Rin} gives $e^{-\alpha Ht}$ decay for some $\alpha>0$. See also the asymptotic expansion \eqref{e0psi.hat.exp} in Theorem \ref{thm:exp}.}
A benefit of our approach, outlined in Section \ref{subsec:approach}, is that it offers a simple and direct way of understanding the global future dynamics of the relevant solutions in the near-FLRW regime, as well as deriving their precise asymptotic behavior at infinity (see Theorems \ref{thm:main}, \ref{thm:exp}). Similar results for the present model have only been previously obtained in spherically symmetry by Costa-Mena \cite{CM} and earlier by Costa-Alho-Nat\'ario \cite{CAN2}, in the near-de Sitter regime (ie. $\psi\sim0$), for a different spatial topology. Precise future asymptotics of cosmological solutions and their various fields have also been investigated in a series of papers by Ashtekar-Bonga-Kesavon \cite{ABK,ABK2,ABK3}, with the aim of developing a theory for gravitational waves that includes $\Lambda>0$. 
Other models that have  generated a lot of interest include the Einstein-Vlasov system with a positive cosmological constant, we refer the reader notably to the works of Andr\'easson-Ringstr\"om \cite{AR} and Ringstr\"om \cite{RinBook}, and to the references therein for a detailed account of related results in the literature.

Our work is in part motivated by the outstanding progress that has been made in the recent years on the understanding of the future dynamics of the FLRW family of solutions to the Einstein-Euler system with a positive cosmological constant and linear equation of state $p=c_s^2\rho$, $c_s\in[0,1]$. Global future stability was first established, for $0<c_s^2<\frac{1}{3}$, by Rodnianski-Speck \cite{RS1} in the irrotational case and by Speck \cite{Sp1} in the presence of vorticity. Remarkably, their work demonstrated that an accelerated expansion can silence fluid degeneracies, like shocks, which are in general expected to occur otherwise.\footnote{See for example Christodoulou's breakthrough work \cite{Christ} on the formation of shocks for the relativistic Euler equations in Minkowski spacetime.}\footnote{Fluid stabilisation, as an effect of an accelerated expansion, was first rigorously established by Brauer-Rendall-Reula \cite{BRR} for Newtonian cosmological models.} The dust case $c_s=0$ was treated in \cite{HS} by Had\v zi\'c-Speck. L\"ubbe--Valiente-Kroon showed that the radiation case $c_s^2=\frac{1}{3}$ can be handled via conformal methods \cite{LK}, in the spirit of Friedrich's original work \cite{Fr}, see also \cite{V} for an extensive account of the applicability of their methods. More recently, Oliynyk introduced a new conformal approach \cite{Ol1} that provides a uniform treatment of future stability in the parameter range $0<c_s^2\leq \frac{1}{3}$. 

Prior to these future stability results, analytic solutions for the Einstein-Euler system with a positive cosmological constant and linear equation of state $p=c_s^2\rho$ were constructed by Rendall \cite{Ren}, for all $c_s\in[0,1)$, where he also computed their full asymptotic expansions at infinity. In the parameter range $0\leq c_s^2\leq\frac{1}{3}$, the asymptotic behavior of the solutions agrees with that observed in the stability works cited above. However, for $c_s^2>\frac{1}{3}$, he noticed that large spatial gradients might cause an abrupt change in the behavior of the fluid variables, which he called reminiscent of spikes that are found near singularities in Gowdy spacetimes \cite{RW,Rin2}. He then speculated that the latter phenomenon could lead to a blowup of the density contrast,\footnote{This is a parameter used in galaxy formation \cite{Grametal} to indicate local enhancements in matter density. In the context of \cite{Ren}, it is equal to $\nabla_a\rho/\rho$, where $\nabla_a$ stands a spatial derivative.} in the parameter range $\frac{1}{3}<c_s^2<1$.

We should also mention here two related works for the relativistic Euler equations with linear equation of state $p=c_s^2\rho$, in a fixed FLRW background, that have recently appeared. In \cite{Ol2}, Oliynyk proved the future stability of certain homogeneous solutions for $\frac{1}{3}<c_s^2<\frac{1}{2}$, suggesting that a similar result may be possible for the coupled Einstein-Euler system. This is the first future stability result going beyond the radiation threshold $c_s^2=\frac{1}{3}$, however, for a class of background homogeneous solutions having non-zero spatial fluid velocity. The second work by Fajman-Oliynyk-Wyatt \cite{FOW} concerns the future stability of (trivial) homogeneous solutions for the irrotational relativistic Euler equations with linear equation of state $p=c_s^2\rho$, in the parameter range $0<c_s^2<\frac{1}{3}$, in the absence of a cosmological constant ($\Lambda=0$), establishing the first case of fluid stabilization with a non-accelerated expansion outside of dust ($p=0$) \cite{FOW2}.

Apart from solutions with compact spatial topology, Schlue \cite{Sc1} has studied the dynamics of the expanding region in Schwarzschild-de Sitter, where the spatial topology near infinity is the cylindrical $\mathbb{R}\times\mathbb{S}^2$, proving a conditional future stability in vacuum, see also \cite{Sc2}. Additional difficulties in that context include deriving uniform estimates up to and including the timelike infinities, where the perturbed solutions converge to a slowly rotating member of the Kerr-de Sitter family (possibly different at each end). Moreover, estimates have to be propagated from the cosmological horizons in the interior of the expanding region. The data along the former are furnished by the stability of the Kerr-de Sitter exterior region, established by Hintz-Vasy in their breakthrough work \cite{HV}. In spherical symmetry, the analogous problem for the Einstein-Maxwell-scalar field system with a positive cosmological constant has been resolved by Costa-Nat\'ario-Oliveira \cite{CNO2}, for general data along the cosmological horizons where only mild convergence to a Reissner-Nordstr\"om-de Sitter black hole is assumed.

Finally, there have been numerous works by various authors proving exponential decay for linear waves in expanding regions of explicit solutions \cite{AC,CAN,CNO,GY,KS,Sc}, as a prelude to non-linear results. We also refer the reader to the work of Ringstr\"om \cite{Rin3} for an extensive study of solutions to linear wave systems in general cosmological backgrounds with convergent asymptotics.

\subsection{The FLRW model}\label{subsec:model}

The solutions that we are interested in are perturbations of the FLRW model. The latter is defined for all $(t,x)\in(-\infty,+\infty)\times\mathbb{T}^3$ (see Lemma \ref{lem:FLRW}):
\begin{align}\label{FLRW.sol}
{\bf g}_{FLRW}=-dt^2+a^2(t)\sum_{i=1}^3(dx^i)^2,&&\psi_{FLRW}(t)=\mathring{\psi}_{FLRW}+\mathring{\varphi}_{FLRW}\int^t_0\frac{\mathring{a}^3}{a^3(\tau)}d\tau,
\end{align}
where $\mathring{a}=\mathring{a}(0),\mathring{\psi}_{FLRW}=\psi_{FLRW}(0),\mathring{\varphi}_{FLRW}=\partial_t\psi_{FLRW}(0)\in\mathbb{R}$ are given constants representing the FLRW initial data for $a(t),\psi_{FLRW}(t)$, and
\begin{align}\label{aFLRW}
a(t)=\mathring{a}\bigg(\sqrt{\frac{\mathring{\varphi}_{FLRW}^2}{\Lambda}+1}\sinh\sqrt{3\Lambda}t+\cosh\sqrt{3\Lambda}t\bigg)^\frac{1}{3}.
\end{align}
Note that in the absence of a scalar field, ie. $\mathring{\psi}_{FLRW}=\mathring{\varphi}_{FLRW}=0$, \eqref{FLRW.sol} reduces to the vacuum solution:
\begin{align}\label{vac.sol}
{\bf g}_{vacuum}=-dt^2+\mathring{a}^2e^{2Ht}\sum_{i=1}^3(dx^i)^2,&&H=\sqrt{\frac{\Lambda}{3}}.
\end{align}
One can easily check that ${\bf g}_{vacuum}$ describes (up to a multiple) the leading order behavior of ${\bf g}_{FLRW}$ at $+\infty$, while the scalar field $\psi_{FLRW}$ decays exponentially to a constant (see Lemma \ref{lem:FLRW}).

\subsection{Our approach to the dynamical problem}\label{subsec:approach}

The framework we adopt to control the perturbed solutions is inspired by the stable Big Bang formation results, in the $\Lambda=0$ case, that have been achieved in the past few years \cite{FRS,RS2,RS3,RS4,Sp2}. However, the present situation is tremendously simpler, since instead of finite time blowup with anisotropic parameters, we derive the exponential decay of solutions to a near-isotropic state in the future direction. The main ingredients that we borrow from the former works concern the $1+3$ splitting of spacetime relative to a time function satisfying an infinite speed of propagation gauge, namely, a parabolic equation for the lapse of its level sets, and the formulation of the Einstein equations relative to a Fermi propagated orthonormal frame, as an ADM-type system for the frame coefficients and its associated connection coefficients. The benefit of such a gauge is that, as we show in Section \ref{sec:main.est}, all the reduced variables, minus their FLRW values, decay exponentially with rates that are easily readable from the ODE part of the resulting evolutionary system (see Lemma \ref{lem:red.eq.2}). 

More precisely, the perturbed spacetime metric in our setup takes the form 
\begin{align}\label{metric}
{\bf g}=-n^2dt^2+g=-n^2dt^2+g_{ij}dx^idx^j,
\end{align}
summing over repeated lower and upper indices.
Here, $n$ is the lapse of the $t$-foliation, $\Sigma_t$, while the shift vector field is set to zero. $g$ is the induced Riemannian metric on $\Sigma_t$, expressed with respect to a holonomic basis of vector fields $\partial_1,\partial_2,\partial_3$. Note that although $\Sigma_t=\{t\}\times\mathbb{T}^3$ cannot be covered by a single chart of spatial coordinates $x^1,x^2,x^3$, the vector fields $\partial_1,\partial_2,\partial_3$ are well-defined everywhere.

Next, we consider a ${\bf g}$-orthonormal frame $\{e_\mu\}_0^3$, expressed with respect to $\partial_t,\partial_i$ via
\begin{align}\label{ei}
e_I=e_I^i\partial_i,\qquad 
e_0=n^{-1}\partial_t,
\end{align}
which is Fermi propagated along $\partial_t$ according to 
\begin{align}\label{Fermi}
{\bf D}_{e_0}e_I=n^{-1}(e_In)e_0,
\end{align}
where ${\bf D}$ is the Levi-Civita connection of ${\bf g}$.
Note that such a frame is orthonormal and adapted to $\Sigma_t$, ie. $e_I\in T\Sigma_t$, provided it is initially the case. For simplicity, we take $e_1,e_2,e_3$ initially to be the orthonormal basis obtained by applying the Gram-Schmidt process to $\partial_1,\partial_2,\partial_3$ on $\Sigma_0$. The following identity also holds:
\begin{align}\label{De0e0}
{\bf D}_{e_0}e_0=n^{-1}(e_I n)e_I,
\end{align}
where Einstein summation over $I=1,2,3$ is employed here, and similarly for capital indices below, without raising indices to emphasize the fact that $e_I$ is orthonormal. 
The connection coefficients associated to $e_\mu$ are
\begin{align}\label{kIJgammaIJB}
k_{IJ}=-{\bf g}({\bf D}_{e_I}e_0,e_J)=k_{JI},\qquad \gamma_{IJB}={\bf g}({\bf D}_{e_I}e_J,e_B)=g(\nabla_{e_I}e_J,e_B)=-\gamma_{IBJ},
\end{align}
where $k$ is in fact the second fundamental form of $\Sigma_t$ and $\nabla$ is the Levi-Civita connection of $g$. 

We normalize the time function $t$ by requiring that the lapse of its level sets equals
\begin{align}\label{parab.gauge}
n-1=\text{tr}k_{FLRW}-\text{tr}k,
\end{align}
where $\text{tr}k_{FLRW}$ is the mean curvature of $\Sigma_t$ in the FLRW background spacetime, see \eqref{red.var.FLRW}.
The condition \eqref{parab.gauge} leads to a parabolic equation for the lapse $n$, see \eqref{n.eq}, which is well-posed in the future direction.
\begin{remark}\label{rem:parab.gauge}
Although parabolic/elliptic gauges are not easily localized, we find the choice \eqref{parab.gauge} convenient, since, as it turns out for us, the lapse $n$ decays exponentially to $1$, see Theorem \ref{thm:main}. Together with the gauge condition \eqref{parab.gauge}, the latter results in an easy read of the asymptotic behaviors of the above variables at infinity, simply by looking at the left-hand sides of the evolution equations in Lemma \ref{lem:red.eq.2}. 
%Interestingly, the constant-mean-curvature foliation $\text{tr}k=\text{tr}k_{FLRW}$, which has been successfully used in \cite{RS2,RS3}, does not seem to give a decaying lapse at infinity.
%We should also note that in \cite{RS2,RS3}, the use of a parabolic gauge appears with the opposite sign in the right-hand side compared to \eqref{parab.gauge}, making the resulting equation for $n$ well-posed in the past direction, since these works focus on the past dynamics towards Big Bang singularities.  
We should note that the idea of using a parabolic gauge in the present context is not new, see for example \cite{LEUW}, where the authors studied the future asymptotics of solutions to the Einstein-Euler equations with linear equation of state, using an inverse mean-curvature-flow type of equation for the lapse of their time slices. 
\end{remark}
After expressing the Einstein equations \eqref{EE}-\eqref{divT} in the above framework (see Lemmas \ref{lem:red.eq}, \ref{lem:red.eq.2}), we prove the global existence of the perturbed solutions via a bootstrap argument for an $H^N(\Sigma_t)$ energy of the variables $k_{IJ},\gamma_{IJB},e^i_I,n,e_0\psi,e_I\psi$ minus their FLRW values, see \eqref{Boots}. Importantly, our bootstrap assumptions are consistent with the exponential in time decay of all the corresponding differences. Then, we use the reduced system of equations satisfied by these variables to obtain an improvement of our bootstrap assumptions. Together with a standard continuation argument for (locally in time) well-posed systems, the latter implies the global existence of the perturbed solutions and the validity of the bootstrapped estimates for all future time, see Proposition \ref{prop:cont.arg}. Subsequently, we use this result to refine our estimates and derive the precise asymptotic behaviors of the spacetime metric and scalar field at infinity (Section \ref{sec:asym.inf}), stated in \eqref{thm.main.est3}-\eqref{thm.main.est4} below. Moreover, we compute the first terms in the asymptotic expansions of all reduced variables (contained in Theorem \ref{thm:exp}).

\subsection{The main results}\label{subsec:thm}

Now that we have introduced the key variables that we use to control the perturbed solution, we can precisely state our main theorem.
\begin{theorem}\label{thm:main}
Let $(\Sigma_0=\mathbb{T}^3,\mathring{g},\mathring{k},\mathring{\psi},\mathring{\varphi})$ denote the perturbed initial data around a fixed FLRW solution \eqref{FLRW.sol}-\eqref{aFLRW}, satisfying the constraints \eqref{Hamconst.intro}-\eqref{momconst.intro}, and let $(\mathcal{M},{\bf g},\psi)$ be the corresponding maximal solution to \eqref{EE}-\eqref{divT}. Consider the reduced variables introduced in the previous subsection and let $\widehat{k}_{IJ},\widehat{\gamma}_{IJB},\widehat{e}^i_I,\widehat{n},\widehat{e_0\psi},\widehat{e_I\psi},\widehat{\psi}$ denote the corresponding differences, after subtracting their FLRW values, see \eqref{red.var.FLRW}, \eqref{diff.var}. Assume that their initial data, induced by $(\mathring{g},\mathring{k},\mathring{\psi},\mathring{\varphi})$ and our choice of $e_I$ on $\Sigma_0$, are sufficiently small in $H^N(\Sigma_0)$, for some $N\ge4$ (see Section \ref{subsec:norms} for the precise definition of the norms): 
\begin{align}\label{thm.main.init}
\|\widehat{k}\|_{H^N(\Sigma_0)}^2
+\|\widehat{\gamma}\|_{H^N(\Sigma_0)}^2+\|\widehat{e}\|_{H^N(\Sigma_0)}^2+\|\widehat{n}\|_{H^N(\Sigma_0)}^2+\|\widehat{e\psi}\|_{H^N(\Sigma_0)}^2+\|\widehat{\psi}\|_{H^N(\Sigma_0)}^2=:\mathring{\varepsilon}^2.
\end{align}
Then, the perturbed solution is globally defined in the future of $\Sigma_0$, relative to the time function $t$ normalized by \eqref{parab.gauge}, and the following estimate holds true:
\begin{align}\label{thm.main.est}
e^{2Ht}\big\{\|\widehat{k}\|_{H^N(\Sigma_t)}^2
+\|\widehat{\gamma}\|_{H^N(\Sigma_t)}^2+\|\widehat{e}\|_{H^N(\Sigma_t)}^2
+e^{Ht}\|\widehat{n}\|_{H^N(\Sigma_t)}^2+\|\widehat{e\psi}\|_{H^N(\Sigma_t)}^2\big\}\leq &\,C\mathring{\varepsilon}^2,
\end{align}
for all $t\in[0,+\infty)$, where the constant $C>0$ depends on $N$ and the FLRW background solution. Also, the following refined estimates for $\widehat{k},\widehat{n},\widehat{e_0\psi}$ are valid:
\begin{align}\label{thm.main.est2}
\|\widehat{k}\|_{C^{N-4}(\Sigma_t)}+\|\widehat{n}\|_{C^{N-4}(\Sigma_t)}+\|\widehat{e_0\psi}\|_{C^{N-4}(\Sigma_t)}\leq C\mathring{\varepsilon}e^{-2Ht}.
\end{align}
Moreover, the components of the spacetime metric \eqref{metric} and the scalar field satisfy:
\begin{align}\label{thm.main.est3}
\|n-1\|_{C^{N-4}(\Sigma_t)}\leq C\mathring{\varepsilon} e^{-2Ht},&\quad \|g_{ij}-e^{2Ht}g_{ij}^\infty(x)\|_{C^{N-4}(\Sigma_t)}\leq C,\quad
 \|\psi-\psi^\infty(x)\|_{C^{N-4}(\Sigma_t)}\leq Ce^{-2Ht},
\end{align}
for all $t\in[0,+\infty)$, where the limiting functions $g^\infty_{ij}(x),\psi^\infty(x)\in C^{N-4}(\mathbb{T}^3)$ are close to their FLRW values (see Lemma \ref{lem:FLRW}):
\begin{align}\label{thm.main.est4}
\bigg\|g^\infty_{ij}(x)-\mathring{a}^2\bigg(\frac{1}{2}\sqrt{\frac{\mathring{\varphi}_{FLRW}^2}{\Lambda}+1}+\frac{1}{2}\bigg)^\frac{2}{3}\delta_{ij}\bigg\|_{C^{N-4}(\mathbb{T}^3)}\leq C\mathring{\varepsilon},\qquad \|\psi^\infty(x)-\psi_{FLRW}^\infty\|_{C^{N-4}(\mathbb{T}^3)}\leq C\mathring\varepsilon.
\end{align}
\end{theorem}
\begin{proof}
It is the combination of Proposition \ref{prop:cont.arg}, Lemma \ref{lem:ref.est}, and Proposition \ref{prop:metric.asym.inf}.
\end{proof}
Although the previous theorem establishes the global existence of perturbed solutions, in the future direction, with adequate information on their asymptotic behavior, the stated estimates do not capture the whole set of asymptotic data at infinity. For example, the asymptotic data of $\psi$ at infinity consists of two functions, while in Theorem \ref{thm:main} only one of the two is identified, namely, $\psi^\infty(x)$. Nevertheless, Theorem \ref{thm:main} can be used, together with the reduced system of equations in Lemma \ref{lem:red.eq.2}, to obtain the full asymptotic expansions of all variables. In the present paper, we only compute the first couple of terms in the expansions that suffice to identify the asymptotic data of all variables. These are in accordance with the expansions of the analytic solutions constructed in \cite{Ren}, where the degrees of freedom and asymptotic constraints at infinity are also studied.
%For the exact expansions, degrees of freedom, and asymptotic constraints at infinity, we refer the reader to the work of Rendall \cite{Ren}.\footnote{Note that in \cite{Ren}, the author considers a geodesic foliation, ie. $n=1$. However, the first few terms in the expansions should be the same, since our estimates imply that $\|n-1\|_{C^{N-4}(\Sigma_t)}\leq C\mathring{\varepsilon}e^{-2t}$.}

The proof of the next theorem is found in the end of Section \ref{subsec:red.exp}.
\begin{theorem}\label{thm:exp}
Consider a perturbed solution furnished by Theorem \ref{thm:main} with $N\ge6$. The variables $\widehat{k}_{IJ},\widehat{\gamma}_{IJB}$, $\widehat{e}^i_I,\widehat{n},\widehat{e_0\psi},\widehat{e_i\psi}$ admit the following finite expansions:
\begin{align}
\label{n.hat.exp}
\|\widehat{n}-\widehat{n}^\infty(x)e^{-2Ht}\|_{C^{N-6}(\Sigma_t)}\leq&\,C\mathring{\varepsilon}e^{-4Ht},\\
\label{eIi.hat.exp}\|\widehat{e}_I^i-(\widehat{e}_I^i)^\infty(x)e^{-Ht}\|_{C^{N-4}(\Sigma_t)}\leq&\,C\mathring{\varepsilon}e^{-3Ht},\\
\label{gamma.hat.exp}\|\widehat{\gamma}_{IJB}-\widehat{\gamma}_{IJB}^\infty(x)e^{-Ht}\|_{C^{N-5}(\Sigma_t)}\leq&\,C\mathring{\varepsilon}e^{-3Ht},\\
\label{k.hat.exp}\|\widehat{k}_{IJ}-F_{\widehat{k}}(x)e^{-2Ht}-\widehat{k}_{IJ}^\infty(x) e^{-3Ht}\|_{C^{N-6}(\Sigma_t)}\leq&\,C\mathring{\varepsilon}e^{-4Ht},\\
\label{e0psi.hat.exp}\|\widehat{e_0\psi}- F_{\widehat{e_0\psi}}(x)e^{-2Ht}-(\widehat{e_0\psi})^\infty(x)e^{-3Ht}\|_{C^{N-6}(\Sigma_t)}\leq&\,C\mathring{\varepsilon}e^{-4Ht},\\
\label{eIpsi.hat.exp}\|\widehat{e_I\psi}-(\widehat{e_I\psi})^\infty(x)e^{-Ht}\|_{C^{N-5}(\Sigma_t)}\leq&\,C\mathring{\varepsilon}e^{-3Ht},
\end{align}
where 
\begin{align}
\label{infty.forcing.k.hat}
F_{\widehat{k}}(x)=&\,\frac{1}{H}\big\{(e_C^a)^\infty\partial_a\widehat{\gamma}_{IJC}^\infty-(e_I^a)^\infty\partial_a\widehat{\gamma}_{CJC}^\infty
-\frac{2}{3}\delta_{IJ}(e_C^a)^\infty\partial_a\widehat{\gamma}_{DDC}^\infty-\widehat{\gamma}_{CID}^\infty\widehat{\gamma}^\infty_{DJC}-\widehat{\gamma}_{IJD}^\infty\widehat{\gamma}_{CCD}^\infty\\
\notag&+\frac{1}{3}\delta_{IJ}(\widehat{\gamma}_{CED}^\infty\widehat{\gamma}^\infty_{DEC}+\widehat{\gamma}_{EED}^\infty\widehat{\gamma}_{CCD}^\infty)-(\widehat{e_I\psi})^\infty(\widehat{e_J\psi})^\infty+\frac{1}{3}\delta_{IJ}(\widehat{e_C\psi})^\infty(\widehat{e_C\psi})^\infty\big\}(x),\\
\label{infty.forcing.e0psi.hat} F_{\widehat{e_0\psi}}(x)=&\,\frac{1}{H}\big\{(e_C^a)^\infty\partial_a(\widehat{e_C\psi})^\infty-\widehat{\gamma}^\infty_{CCD}(\widehat{e_D\psi})^\infty\big\}(x),\\
\label{eIi.infty}(e_I^i)^\infty(x):=&\,(\widehat{e}_I^i)^\infty(x)+\delta^i_I\mathring{a}^{-1}\bigg(\frac{1}{2}\sqrt{\frac{\mathring{\varphi}_{FLRW}^2}{\Lambda}+1}+\frac{1}{2}\bigg)^{-\frac{1}{3}},
\end{align}
for all $t\in[0,+\infty)$, where
$\widehat{n}^\infty(x),\widehat{k}_{IJ}^\infty(x),(\widehat{e_0\psi})^\infty(x)\in C^{N-6}(\mathbb{T}^3)$,
$\widehat{\gamma}_{IJB}^\infty(x),(\widehat{e_I\psi})^\infty(x)\in C^{N-5}(\mathbb{T}^3)$, $(\widehat{e}_I^i)^\infty(x)\in C^{N-4}(\mathbb{T}^3)$, are the renormalised limits at $+\infty$ of the corresponding variables or algebraic expressions above, satisfying
\begin{align}\label{hat.lim.est}
\begin{split}
\|\widehat{n}^\infty(x)\|_{C^{N-6}(\mathbb{T}^3)},\|(\widehat{e}_I^i)^\infty(x)\|_{C^{N-4}(\mathbb{T}^3)},\|\widehat{\gamma}_{IJB}^\infty(x)\|_{C^{N-5}(\mathbb{T}^3)},\|\widehat{k}_{IJ}^\infty(x)\|_{C^{N-6}(\mathbb{T}^3)},\\
\|(\widehat{e_0\psi})^\infty(x)\|_{C^{N-6}(\mathbb{T}^3)},\|(\widehat{e_I\psi})^\infty(x)\|_{C^{N-5}(\mathbb{T}^3)}\leq C\mathring{\varepsilon}.
\end{split}
\end{align}
\end{theorem}
\begin{remark}\label{rem:asym.data}
The limiting functions $\widehat{n}^\infty(x),(\widehat{e}_I^i)^\infty(x),\widehat{\gamma}_{IJB}^\infty(x),\widehat{k}_{IJ}^\infty(x),(\widehat{e_0\psi})^\infty(x),(\widehat{e_I\psi})^\infty(x)$ are the `free data' of each corresponding reduced variable at infinity. However, they are not all independent. For example, the functions $\widehat{\gamma}_{IJB}^\infty(x)$ can be computed from $(\widehat{e}_I^i)^\infty(x)$ and their first derivatives. After taking into account such dependencies, one obtains the asymptotic data of the perturbed solution, expressed in the above gauge.
\end{remark}
\begin{remark}\label{rem:k.free}
The asymptotic data of $\widehat{k}_{IJ},\widehat{e_0\psi}$ at infinity are captured by the coefficients of the second order terms in their corresponding expansions, due to the fact that the homogeneous branches of solutions to the equations \eqref{k.hat.eq}, \eqref{e0psi.hat.eq} are of order $e^{-3Ht}$, while the contributions from the corresponding RHSs decay slower, giving the leading order terms $F_{\widehat{k}}(x)e^{-2Ht},F_{\widehat{e_0\psi}}(x)e^{-2Ht}$.
\end{remark}
As we discussed in the beginning of the introduction, we are also interested in the interpretation of the massless-scalar field model as an irrotational fluid with linear equation of state $p=\rho$. More generally, a solution to the Einstein-Euler system with linear equation of state $p=c_s^2\rho$, $c_s\in[0,1]$, in the irrotational case (see for example \cite{RS1}), possesses a fluid potential $\psi$ such that 
\begin{align}\label{rho}
\rho=\frac{1}{1+c_s^2}[-{\bf g}(\mathrm{grad}\psi,\mathrm{grad}\psi)]^{\frac{1+c_s^2}{2c_s^2}},\qquad u=-\frac{\mathrm{grad}\psi}{\sqrt{-{\bf g}(\mathrm{grad}\psi,\mathrm{grad}\psi)}},
\end{align}
where $\mathrm{grad}\psi$ is the gradient of $\psi$ and $u$ is the unit fluid speed, ${\bf g}(u,u)=-1$. In our framework ($c_s=1$),
\begin{align}\label{rho2}
\mathrm{grad}\psi=-(e_0\psi) e_0+(e_I\psi) e_I, \qquad2\rho=(e_0\psi)^2-(e_I\psi)e_I\psi.
\end{align}
%
%We should note that 
This is the limiting situation where the sound speed $c_s$ is equal to the speed of light. 
%As such, degeneracies like shocks are not expected to form (and indeed they do not), so one should be careful how to extrapolate any features to the $c_s<1$ regime. 

For the solutions constructed in \cite{Ren}, in the parameter range $\frac{1}{3}<c_s<1$,
the author required that the spatial part of $u$, ie. its projection $\overline{u}:={\bf g}(u,e_I)e_I$ onto $\Sigma_t$, is nowhere vanishing, to avoid any inhomogeneous features in $\rho$ that could lead to a blowup in the density contrast. For the present model of study, the non-vanishing of $\overline{u}$ near infinity is equivalent to the non-vanishing of the spatial gradient of $\psi$, $(e_I\psi) e_I$. Since $\psi(t,\cdot)$ is a scalar function from $\mathbb{T}^3\to\mathbb{R}$, it necessarily has critical points, for any $t$, where its spatial gradient vanishes. Nevertheless, for general asymptotic data, there will be spacetime regions near infinity where $\psi^\infty(x)$ has no critical points. In the next corollary, we show that even in such regions, the interpretation of our perturbed solutions as irrotational fluids breaks down, due to the spacetime gradient of $\psi$ being spacelike. For the Einstein-Euler system, the corresponding irrotational solutions would degenerate at the first point where the fluid speed $u$, ie. the normalized spacetime gradient of $\psi$, becomes null. Evidently, such a point exists, since grad$\psi$ starts out timelike (on $\Sigma_0$) and it is in general spacelike near infinity (according to Corollary \ref{cor:fluid}). It is tempting to speculate that a similar breakdown would also occur in near-FLRW solutions, at least in the irrotational case, for $c_s<1$ but close to $1$. However, we refrain from making any claims due to the special nature of the limiting case $c_s=1$. 
\begin{corollary}\label{cor:fluid}
Let $({\bf g},\psi)$ be a solution to \eqref{EE}-\eqref{divT} as in Theorem \ref{thm:main}, for $N\ge5$. Then the norm of the gradient of the scalar field satisfies the estimate:
\begin{align}\label{grad.psi.exp}
\big\|-(e_0\psi)^2+(e_I\psi)e_I\psi-(e_I^a)^\infty(x)(e^b_I)^\infty(x)\partial_a\psi^\infty(x)\partial_b\psi^\infty(x)e^{-2Ht}\big\|_{C^{N-5}(\Sigma_t)}\leq Ce^{-4Ht},
\end{align}
for all $t\in[0,+\infty)$, where $(e^a_I)^\infty(x),\psi^\infty(x)$ are as in \eqref{thm.main.est3}, \eqref{eIi.infty}. In particular, $\mathrm{grad}\psi(t,x)$ becomes spacelike for $t$ sufficiently large, how large depending on the specific $x\in\mathbb{T}^3$, provided $x$ is not a critical point of the limit function $\psi^\infty:\mathbb{T}^3\to\mathbb{R}$.
\end{corollary}
\begin{proof}
The estimate \eqref{thm.main.est2}, together with the formulas \eqref{red.var.FLRW}, \eqref{diff.var} for $(e_0\psi)_{FLRW},(\widehat{e_0\psi})$, imply that $\|e_0\psi\|_{C^{N-4}(\mathbb{T}^3)}^2\leq C e^{-4Ht}$. On the other hand, the spatial gradient of $\psi$ equals
$(e_I\psi)e_I\psi=e_I^ae_I^b\partial_a\psi\partial_b\psi$. The estimate \eqref{thm.main.est} and the formulas \eqref{red.var.FLRW}, \eqref{diff.var} for $({e}_I^a)_{FLRW},\widehat{e}^a_I$, imply that $|e_I^a|\leq Ce^{-Ht}$. Then, the estimate \eqref{thm.main.est3} gives
\begin{align*}
|e_I^ae_I^b\partial_a\psi\partial_b\psi-e_I^ae_I^b\partial_a\psi^\infty(x)\partial_b\psi^\infty(x)|\leq Ce^{-4Ht}.
\end{align*}
Moreover, using the expansion \eqref{eIi.hat.exp} for $\widehat{e}_I^a$, the formulas \eqref{eIi.infty} for $(e^a_I)^\infty$ and \eqref{red.var.FLRW} for $(e_I^a)_{FLRW}$, we have 
\begin{align*}
|e_I^ae_I^b\partial_a\psi^\infty(x)\partial_b\psi^\infty(x)-(e_I^a)^\infty(x)(e_I^b)^\infty(x)\partial_a\psi^\infty(x)\partial_b\psi^\infty(x)|\leq Ce^{-4Ht}.
\end{align*}
The asserted estimate \eqref{grad.psi.exp} follows from the above bounds, triangular inequality, and the regularity of the involved variables $\psi^\infty(x),(e_I^a)^\infty(x)\in C^{N-4}(\mathbb{T}^3)$. 
\end{proof}
Interestingly, had grad$\psi$ remained forever timelike, it would make sense to use $\psi$ as a time function. This option was explored by Alho--Mena--Valiente-Kroon \cite{AMK} (see also other references therein), in the context of the Einstein-Friedrich-nonlinear scalar field system. Unfortunately, for the solutions at hand, according to Corollary \ref{cor:fluid}, such a gauge would in general degenerate near infinity.

\subsection{Outline of the paper}

In Section \ref{sec:setup}, we derive the equations satisfied by the variables introduced in Section \ref{subsec:approach}, as well as the resulting equations after subtracting the background FLRW variables \eqref{red.var.FLRW}, see Lemmas \ref{lem:red.eq}, \ref{lem:red.eq.2}. Next, we introduce the norms and bootstrap assumptions in Section \ref{sec:norms.boots}. Then, in Section \ref{sec:main.est}, we derive the energy estimates that are needed to close the bootstrap argument and obtain a global solution. The standard continuation argument, described in the proof of Proposition \ref{prop:cont.arg}, relies on the well-posedness properties of the reduced system of equations used in the present paper. Although the actual system is not well-posed per se (a well-known problem of ADM-type systems), the required properties can be derived by studying a modified system, as it has been done in various past works, see Appendix \ref{sec:app}. Finally, in Section \ref{sec:asym.inf}, we use the obtained global solution and the bootstrapped estimates to derive refined estimates and the precise asymptotic behavior of the perturbed solutions at infinity, thus, completing the proofs of Theorems \ref{thm:main}, \ref{thm:exp}.

\subsection{Notation}

We adopt the following notational conventions without any further mentioning below.
\begin{itemize}

\item $H=\sqrt{\frac{\Lambda}{3}}$.

\item We will use $C>0$ to denote a generic constant that depends on the number $N$ of derivatives in our norms and the FLRW background we are perturbing about. Also, it will be allowed to change from one line to the next. 

\item We will use the symbol $\mathcal{O}(e^{Bt})$, $B\in\mathbb{R}$, to denote smooth homogeneous functions $\mathbb{R}\to\mathbb{R}$, depending solely on the background FLRW solution \eqref{FLRW.sol}, that satisfy $|\partial_t^N\mathcal{O}(e^{Bt})|\leq C_{N,B}e^{Bt}$, for any $N\in\mathbb{N}$, where the constant $C_{N,B}$ depends on $N,B$. 

\item Latin indices $a,b,i,j,A,B,I,J$  
range over $\{1,2,3\}$. Small letters correspond to coordinate vector fields, while capital letters are reserved for the spatial orthonormal frame. In few instances we use Greek letters $\alpha,\beta,\mu,\nu$, mainly to define our curvature conventions, ranging over $\{0,1,2,3\}$.

\item We use Einstein summation for repeated indices. Whenever a sum is computed relative to the spatial orthonormal frame, we do not raise indices, e.g. $(e_In)e_I=\sum_{I=1}^3(e_In)e_I$.

\item The Riemann curvature ${\bf Riem}$, Ricci curvature ${\bf Ric}$, and scalar curvature ${\bf R}$ of ${\bf g}$ are defined as follows:
\begin{align} \label{curv.bf.g}
\notag	{\bf Riem}(e_{\alpha},e_{\beta},e_{\mu},e_{\nu}):=&\,{\bf g}
	({\bf D}^2_{e_{\alpha} e_{\beta}} e_\nu 
		- 
		{\bf D}^2_{e_{\beta} e_{\alpha}} e_\nu,
		e_\mu ), \\
	{\bf Ric}(e_{\alpha},e_{\beta})
	 :=&-{\bf Riem}(e_{\alpha},e_0,e_{\beta},e_0)+{\bf Riem}(e_{\alpha},e_I,e_{\beta},e_I),\\
\notag	{\bf R }
	 :=&-{\bf Ric}(e_0,e_0)+{\bf Ric}(e_I,e_I),
\end{align}
where ${\bf D}^2_{e_{\alpha} e_{\beta}}e_\nu:={\bf D}_{e_\alpha}({\bf D}_{e_\beta}e_\nu)-{\bf D}_{{\bf D}_{e_\alpha}e_{\beta}}e_\nu$.
The corresponding curvature tensors of $g$, denoted by $Riem$, $Ric$, $R$, are defined analogously.

\end{itemize}

\subsection{Acknowledgements}

I would like to thank Todd Oliynyk, Hans Ringstr\"om, Volker Schlue and Jared Speck for useful communications. Also, I would like to thank Artur Alho, Jos\'e Nat\'ario and Jo\~ao Costa for bringing additional references to my attention. 
The author gratefully acknowledges the support of the \texttt{ERC grant 714408 GEOWAKI}, under the European Union's Horizon 2020 research and innovation program.

\section{Setting up the perturbed problem}\label{sec:setup}

In this section we formulate the Einstein equations in the setup introduced in Section \ref{subsec:approach}. We then compute the background FLRW variables and the resulting equations for the variables that measure the closeness of the perturbed solution to the background.

\subsection{The reduced Einstein equations}

Recall our gauge choices and variables in \eqref{metric}-\eqref{parab.gauge}. The Einstein-massless-scalar field system \eqref{EE}-\eqref{divT} reduces to:
\begin{lemma}\label{lem:red.eq}
The variables $k_{IJ},\gamma_{IJB},e^i_I,n$ satisfy the evolution equations: 
\begin{align}
\label{k.eq}e_0k_{IJ}+(n-1-\mathrm{tr}k_{FLRW})k_{IJ}=&-n^{-1} e_I e_J n+e_C \gamma_{IJC}-e_I\gamma_{CJC}\\
\notag&+n^{-1}\gamma_{IJC} e_C n-\gamma_{CID}\gamma_{DJC}-\gamma_{IJD}\gamma_{CCD}-\Lambda\delta_{IJ}-e_I\psi e_J\psi,\\
\label{gamma.eq}e_0\gamma_{IJB}-k_{IC} \gamma_{CJB}=&\,e_B k_{IJ}-e_J k_{BI}-k_{IC}\gamma_{BJC}-k_{CJ} \gamma_{BIC}+k_{IC} \gamma_{JBC}
+k_{BC} \gamma_{JIC}\\
&+n^{-1}(e_B n)k_{JI}-n^{-1}(e_J n)k_{BI},\notag\\
\label{eIi.eq}e_0e_I^i=&\,k_{IC} e_C^i,\\
\label{n.eq}\partial_tn-e_C e_C n=&-\gamma_{CCD}e_Dn-nk_{CD}k_{CD}+n\Lambda-n(e_0\psi)^2+\partial_t\mathrm{tr}k_{FLRW}
\end{align}
The derivatives of the scalar field $e_0\psi,e_I\psi$ satisfy the evolution equations:
\begin{align}
\label{e0psi.eq}e_0(e_0\psi)+(n-1-\mathrm{tr}k_{FLRW})e_0\psi=&\,e_C(e_C\psi)-\gamma_{CCD}e_D\psi+n^{-1}(e_Cn)e_C\psi,\\
\label{eIpsi.eq}e_0(e_I\psi)-k_{IC}e_C\psi=&\,e_I(e_0\psi)+n^{-1}(e_In)e_0\psi.
\end{align}
Also, the following constraint equations hold:
\begin{align}
\label{Hamconst}2e_C \gamma_{DDC}
-
\gamma_{CDE}\gamma_{EDC}
-
\gamma_{CCD}\gamma_{EED}
=&\,k_{CD}k_{CD}-(n-1-\mathrm{tr}k_{FLRW})^2+2\Lambda+(e_0\psi)^2+e_C\psi e_C\psi,\\
e_C k_{CI}+e_In-  
k_{ID} \gamma_{CCD}
-
k_{CD} \gamma_{CID}=&-e_0\psi e_I\psi.
\label{momconst}
 \end{align}
\end{lemma}
\begin{proof}
To prove \eqref{k.eq} we use the formula \eqref{kIJgammaIJB} for $k_{IJ}$, the propagation relations \eqref{Fermi}, \eqref{De0e0}, and the Einstein equations \eqref{EE}, to compute:
\begin{align}\label{R0I0J}
\notag e_0k_{IJ}=&-{\bf g}({\bf D}_{e_0}({\bf D}_{e_I}e_0),e_J)-n^{-1}(e_Jn){\bf g}({\bf D}_{e_I}e_0,e_0)\\
\notag=&-{\bf g}({\bf D}_{e_0e_I}^2e_0,e_J)-{\bf g}({\bf D}_{{\bf D}_{e_0}e_I}e_0,e_J)\\
\notag=&\,{\bf Riem}(e_I,e_0,e_J,e_0)-{\bf g}({\bf D}_{e_Ie_0}^2e_0,e_J)-n^{-2}(e_In)(e_Jn)\\
=&-{\bf Ric}(e_I,e_J)+{\bf Riem}(e_I,e_C,e_J,e_C)-{\bf g}({\bf D}_{e_I}({\bf D}_{e_0}e_0),e_J)+{\bf g}({\bf D}_{{\bf D}_{e_I}e_0}e_0,e_J)\\
\notag&-n^{-2}(e_In)(e_Jn)\\
\notag=&-\Lambda\delta_{IJ}-e_I\psi e_J\psi+{\bf Riem}(e_I,e_C,e_J,e_C)
-n^{-1}e_Ie_Jn+n^{-1}\gamma_{IJC}e_Cn+k_{IC}k_{CJ},
\end{align}
where in the last line we also used the anti-symmetry of $\gamma_{ICJ}=-\gamma_{IJC}$.
Next, we use the Gauss equations
\begin{align}\label{Gauss}
{\bf Riem}(e_I,e_C,e_J,e_C)=&\,Ric(e_I,e_J)-k_{IC}k_{CJ}+\text{tr}kk_{IJ}
\end{align}
and expand the Ricci term in the last RHS:
\begin{align}\label{RIJ}
\notag Ric(e_I,e_J)=&\, Riem(e_I,e_C,e_J,e_C)=g(\nabla_{e_Ie_C}^2e_C-\nabla_{e_Ce_I}^2e_C,e_J)\\
\notag=&\,g(\nabla_{e_I}(\nabla_{e_C}e_C),e_J)-g(\nabla_{e_C}(\nabla_{e_I}e_C),e_J)-g(\nabla_{\nabla_{e_I}e_C}e_C,e_J)+g(\nabla_{\nabla_{e_C}e_I}e_C,e_J)\\
=&\,g(\nabla_{e_I}(\gamma_{CCD}e_D),e_J)-g(\nabla_{e_C}(\gamma_{ICD}e_D),e_J)-\gamma_{ICD}\gamma_{DCJ}+\gamma_{CID}\gamma_{DCJ}\\
\notag=&\,e_I\gamma_{CCJ}+\gamma_{CCD}\gamma_{IDJ}-e_C\gamma_{ICJ}-\gamma_{ICD}\gamma_{CDJ}-\gamma_{ICD}\gamma_{DCJ}+\gamma_{CID}\gamma_{DCJ}\\
\notag=&-e_I\gamma_{CJC}+e_C\gamma_{IJC}-\gamma_{CID}\gamma_{DJC}-\gamma_{CCD}\gamma_{IJD},
\end{align}
where in the last equality we used the anti-symmetry $\gamma_{IJB}=-\gamma_{IBJ}$.
Plugging \eqref{Gauss}-\eqref{RIJ} into \eqref{R0I0J} and using the gauge condition \eqref{parab.gauge} yields \eqref{k.eq}.

Also, contracting \eqref{Gauss}, \eqref{RIJ} gives
\begin{align}\label{Gauss2}
{\bf R}+2{\bf Ric}(e_0,e_0)=&\,2e_C\gamma_{DDC}-\gamma_{CED}\gamma_{DEC}-\gamma_{CCD}\gamma_{EED}-k_{CD}k_{CD}+(\text{tr}k)^2,
\end{align}
which after using the Einstein equations \eqref{EE}, the gauge condition \eqref{parab.gauge}, and re-arranging the terms, results in the constraint equation \eqref{Hamconst}. The constraint equation \eqref{momconst} is a rewriting of \eqref{momconst.intro} (its version along $\Sigma_t$ instead of the initial hypersurface) using \eqref{parab.gauge}, the homogeneity of $\text{tr}k_{FLRW}$, $e_I\text{tr}k_{FLRW}=0$, and expanding the divergence of $k$ relative to the frame $e_I$.

Taking the trace of \eqref{k.eq} and plugging in \eqref{parab.gauge}, \eqref{Hamconst}, we obtain:
\begin{align*}
&-e_0(n-1-\mathrm{tr}k_{FLRW})-(n-1-\mathrm{tr}k_{FLRW})^2\\
=&-n^{-1} e_C e_C n
+n^{-1}\gamma_{DDC} e_C n
+2e_C \gamma_{DDC}-\gamma_{CED}\gamma_{DEC}-\gamma_{EED}\gamma_{CCD}-e_C\psi e_C\psi-3\Lambda\\
=&-n^{-1} e_C e_C n
+n^{-1}\gamma_{DDC} e_C n+k_{CD}k_{CD}-\Lambda+(e_0\psi)^2-(n-1-\mathrm{tr}k_{FLRW})^2
\end{align*}
Multiplying both sides with $n$ and re-arranging the terms gives the equation \eqref{n.eq} for the lapse.

The evolution equation \eqref{eIi.eq} for the frame coefficients is a direct consequence of the identity \eqref{ei} and the propagation condition \eqref{Fermi}:
\begin{align*}
[e_0,e_I]=&\,{\bf D}_{e_0}e_I-{\bf D}_{e_I}e_0=n^{-1}(e_In)e_0+k_{IC}e_C=n^{-1}(e_In)e_0+k_{IC}e_C^i\partial_i\\
[e_0,e_I]=&\,[n^{-1}\partial_t,e_I]=n^{-1}[\partial_t,e_I]+n^{-1}(e_In)e_0=n^{-1}(\partial_t e^i_I)\partial_i+n^{-1}(e_In)e_0
\end{align*}
The latter commutator identity, together with \eqref{eIi.eq}, also imply the equation \eqref{eIpsi.eq} for $e_I\psi$.

The equation \eqref{e0psi.eq} is the wave equation \eqref{divT} expanded relative to $e_\mu$, using \eqref{De0e0} and the gauge condition \eqref{parab.gauge} for $\text{tr}k$:
\begin{align*}
0=&-{\bf D}_{e_0e_0}^2\psi+{\bf D}^2_{e_Ce_C}\psi\\
=&-e_0e_0\psi+{\bf D}_{{\bf D}_{e_0}e_0}\psi+e_Ce_C\psi-{\bf D}_{{\bf D}_{e_C}e_C}\psi\\
=&-e_0e_0\psi+n^{-1}(e_Cn)e_C\psi+e_Ce_C\psi-\gamma_{CCD}e_D\psi+\text{tr}ke_0\psi
\end{align*}

Finally, for \eqref{gamma.eq} we use the formula \eqref{kIJgammaIJB} for $\gamma_{IJB}$, the propagation condition \eqref{Fermi}, and the Codazzi equations to compute:
\begin{align*}
e_0\gamma_{IJB}=&\,{\bf g}({\bf D}_{e_0}({\bf D}_{e_I}e_J),e_B)+{\bf g}({\bf D}_{e_I}e_J,{\bf D}_{e_0}e_B)\\
=&\,{\bf g}({\bf D}^2_{e_0e_I}e_J,e_B)+{\bf g}({\bf D}_{{\bf D}_{e_0}e_I}e_J,e_B)+n^{-1}(e_Bn)k_{IJ}\\
=&\,{\bf Riem}(e_0,e_I,e_B,e_J)+{\bf g}({\bf D}^2_{e_Ie_0}e_J,e_B)+n^{-2}(e_In)(e_Jn){\bf g}(e_0,e_B)+n^{-1}(e_Bn)k_{IJ}\\
=&\,{\bf Riem}(e_B,e_J,e_0,e_I)+{\bf g}({\bf D}_{e_I}({\bf D}_{e_0}e_J),e_B)-{\bf g}({\bf D}_{{\bf D}_{e_I}e_0}e_J,e_B)+n^{-1}(e_Bn)k_{IJ}\\
=&\,\nabla_{e_B}k_{JI}-\nabla_{e_J}k_{BI}+{\bf g}({\bf D}_{e_I}[n^{-1}(e_Jn)e_0],e_B)+k_{IC}\gamma_{CJB}+n^{-1}(e_Bn)k_{IJ}\\
=&\,\nabla_{e_B}k_{JI}-\nabla_{e_J}k_{BI}-n^{-1}(e_Jn)k_{BI}+k_{IC}\gamma_{CJB}+n^{-1}(e_Bn)k_{IJ}
\end{align*}
We arrive at the desired equation after expanding the covariant derivative terms in the last RHS. This completes the proof of the lemma.
\end{proof}

\subsection{The FLRW background variables}

Following \cite{RS1}, we start by computing the FLRW solution \eqref{FLRW.sol} and its leading order behavior at $+\infty$.
\begin{lemma}\label{lem:FLRW}
A homogeneous solution to \eqref{EE}-\eqref{divT} with metric having a warped product form as in \eqref{FLRW.sol}, is given by \eqref{FLRW.sol}-\eqref{aFLRW}, and moreover, the warping function $a(t)$ and scalar field $\psi_{FLRW}(t)$ satisfy:
\begin{align}\label{FLRW.asym}
\begin{split}
\bigg|\partial^N_ta(t)-\mathring{a}H^N\bigg(\frac{1}{2}\sqrt{\frac{\mathring{\varphi}_{FLRW}^2}{\Lambda}+1}+\frac{1}{2}\bigg)^\frac{1}{3}e^{Ht}\bigg|\leq&\, C_Ne^{-5Ht},\\ 
\bigg|\partial^N_t\bigg\{\partial_t\psi_{FLRW}(t)-2\bigg(\sqrt{\frac{\mathring{\varphi}_{FLRW}^2}{\Lambda}+1}+1\bigg)^{-1}\mathring{\varphi}_{FLRW}e^{-3Ht}\bigg\}\bigg|\leq&\,C_Ne^{-9Ht},
\end{split}
\end{align}
for all $t\in[0,+\infty)$, $N\in\mathbb{N}$, where $\mathring{a}=a(0),\mathring{\varphi}_{FLRW}=\partial_t\psi_{FLRW}(0)$. In particular, the following limits exist:
\begin{align}\label{FLRW.lim}
\lim_{t\rightarrow+\infty}e^{-Ht}a(t)=\mathring{a}\bigg(\frac{1}{2}\sqrt{\frac{\mathring{\varphi}_{FLRW}^2}{\Lambda}+1}+\frac{1}{2}\bigg)^\frac{1}{3},\qquad\lim_{t\rightarrow+\infty}\psi_{FLRW}(t)=:\psi_{FLRW}^\infty.
\end{align}
\end{lemma}
\begin{proof}
By \cite[Lemma 4.2]{RS1}, for $c_s=1$, the warping function equals:
\begin{align}\label{aFLRW2}
a(t)=\mathring{a}\bigg(\sqrt{\frac{\mathring{\rho}_{FLRW}}{\Lambda}+1}\sinh\sqrt{3\Lambda}t+\cosh\sqrt{3\Lambda}t\bigg)^\frac{1}{3},
\end{align}
where $\mathring{\rho}_{FLRW}$ is the initial value of the density $\rho_{FLRW}(t)$. On the other hand, the formula \eqref{rho2} for the FLRW solution reduces to:
\begin{align}\label{rhoFLRW}
\rho_{FLRW}(t)=[\partial_t\psi_{FLRW}(t)]^2.
\end{align}
Hence, $a(t)$ is given by \eqref{aFLRW}. Also, the wave equation \eqref{e0psi.eq} becomes:
\begin{align}\label{e0psi.FLRW}
\partial_t^2\psi_{FLRW}(t)+3\frac{\partial_ta(t)}{a(t)}\partial_t\psi_{FLRW}(t)=0\qquad\Rightarrow\qquad\partial_t[a^3(t)\partial_t\psi_{FLRW}(t)]=0,
\end{align}
which gives the formula \eqref{FLRW.sol} for $\psi_{FLRW}(t)$ after integrating in $[0,t]$. The estimates \eqref{FLRW.asym} are now straightforward consequences of the formulas \eqref{FLRW.sol}, \eqref{aFLRW}. The first limit in \eqref{FLRW.lim} is obvious from \eqref{FLRW.asym}. The scalar field $\psi_{FLRW}(t)$ also has a limit at $+\infty$, since its time derivative is exponentially decaying. 
\end{proof}
Since the FLRW metric \eqref{FLRW.sol} is of the form \eqref{metric}, after appropriate identifications with respect to the coordinate system ($t,x^1,x^2,x^3$),  we may consider the corresponding reduced variables, introduced in Section \ref{subsec:approach}, for the FLRW solution \eqref{FLRW.sol}.
In view of \eqref{aFLRW}, \eqref{FLRW.asym}, these equal:
\begin{align}\label{red.var.FLRW}
\notag (k_{IJ})_{FLRW}=-\delta_{IJ}\frac{\partial_ta(t)}{a(t)}=-\delta_{IJ}H+\mathcal{O}(e^{-6Ht}),\qquad (\gamma_{IJB})_{FLRW}=0,\qquad
n_{FLRW}=1,\\
(e^i_I)_{FLRW}=\delta^i_I[a(t)]^{-1}=\delta^i_I\mathring{a}^{-1}\bigg(\frac{1}{2}\sqrt{\frac{\mathring{\varphi}_{FLRW}^2}{\Lambda}+1}+\frac{1}{2}\bigg)^{-\frac{1}{3}}e^{-Ht}+\delta^i_I\mathcal{O}(e^{-7Ht}),\qquad\psi_{FLRW}=\mathcal{O}(1),\\
\notag (e_0\psi)_{FLRW}=2\bigg(\sqrt{\frac{\mathring{\varphi}_{FLRW}^2}{\Lambda}+1}+1\bigg)^{-1}\mathring{\varphi}_{FLRW}e^{-3Ht}+\mathcal{O}(e^{-9Ht}),\qquad (e_I\psi)_{FLRW}=0,
\end{align}
where we recall that we use the notation $\mathcal{O}(e^{Bt})$ to denote smooth functions $\mathbb{R}\to\mathbb{R}$ depending only on the FLRW parameters, and satisfying $|\partial_t^N\mathcal{O}(e^{Bt})|\leq C_{N,B}e^{Bt}$, for all $t\in[0,+\infty)$ and $N\in\mathbb{N}$.

\subsection{The resulting equations for the perturbed variables minus the FLRW variables}

First, notice that combining the gauge condition \eqref{parab.gauge} with \eqref{red.var.FLRW} gives the following relation:
\begin{align}\label{trk}
\text{tr}k=-\big[3H+\mathcal{O}(e^{-2Ht})+n-1\big]
\end{align}
In this subsection we derive the equations satisfied by the differences:
\begin{align}\label{diff.var}
\notag\widehat{k}_{IJ}=k_{IJ}-\frac{1}{3}\delta_{IJ}\text{tr}k\overset{\eqref{parab.gauge},\eqref{red.var.FLRW}}{=}k_{IJ}-(k_{IJ})_{FLRW}+\frac{1}{3}\delta_{IJ}(n-1),\qquad\widehat{n}=n-n_{FLRW}=n-1\\ 
\widehat{\gamma}_{IJB}=\gamma_{IJB}-(\gamma_{IJB})_{FLRW}=\gamma_{IJB},\qquad\widehat{e}^i_I=e^i_I-(e^i_I)_{FLRW},\qquad\widehat{\psi}=\psi-\psi_{FLRW}\\
\notag\widehat{e_0\psi}=e_0\psi-(e_0\psi)_{FLRW}=e_0\psi+\mathcal{O}(e^{-3Ht})\qquad\widehat{e_I\psi}=e_I\psi-(e_I\psi)_{FLRW}=e_I\psi
\end{align}
using Lemma \ref{lem:red.eq}. Note that since the background variables \eqref{red.var.FLRW} are homogeneous, it holds $e_In=e_I\widehat{n}$ etc.
\begin{lemma}\label{lem:red.eq.2}
The variables $\widehat{k}_{IJ},\widehat{\gamma}_{IJB},\widehat{e}^i_I,\widehat{n}$ satisfy the evolution equations: 
\begin{align}
\label{k.hat.eq}\partial_t\widehat{k}_{IJ}+\big[3H+\mathcal{O}(e^{-2Ht})+\widehat{n}\big]\widehat{k}_{IJ}=&-e_I e_J \widehat{n}+\frac{1}{3}\delta_{IJ}e_C e_C \widehat{n}+ne_C \widehat{\gamma}_{IJC}-ne_I\widehat{\gamma}_{CJC}\\
\notag&-\frac{2}{3}\delta_{IJ}ne_C \widehat{\gamma}_{DDC}
+\gamma_{IJC} e_C n-n\gamma_{CID}\gamma_{DJC}-n\gamma_{IJD}\gamma_{CCD}\\
\notag&-\frac{1}{3}\delta_{IJ}\big\{\gamma_{DDC} e_C n-n\gamma_{CED}\gamma_{DEC}-n\gamma_{EED}\gamma_{CCD}\big\}\\
\notag&-ne_I\psi e_J\psi+\frac{1}{3}\delta_{IJ}ne_C\psi e_C\psi-\big[3H+\mathcal{O}(e^{-2Ht})+\widehat{n}\big]\widehat{n}\widehat{k}_{IJ},\\
\label{gamma.hat.eq}\partial_t\widehat{\gamma}_{IJB}+\big[H+\mathcal{O}(e^{-2Ht})+\frac{1}{3}\widehat{n}\big]\widehat{\gamma}_{IJB}=&\,ne_B \widehat{k}_{IJ}-ne_J \widehat{k}_{BI}-\frac{1}{3}\delta_{IJ}ne_B\widehat{n}+\frac{1}{3}\delta_{IB}ne_J\widehat{n}
-n\widehat{k}_{IC}\gamma_{BJC}\\
\notag&-n\widehat{k}_{CJ} \gamma_{BIC}+n\widehat{k}_{IC} \gamma_{JBC}
+n\widehat{k}_{BC} \gamma_{JIC}+n\widehat{k}_{IC} \gamma_{CJB}\\
\notag&+(e_B \widehat{n})k_{JI}-(e_J \widehat{n})k_{BI}-\big[H+\mathcal{O}(e^{-2Ht})+\frac{1}{3}\widehat{n}\big]\widehat{n}\widehat{\gamma}_{IJB},\\
\label{eIi.hat.eq}\partial_t \widehat{e}_I^i+\big[H+\mathcal{O}(e^{-2Ht})+\frac{1}{3}\widehat{n}\big]\widehat{e}^i_I=&\,n\widehat{k}_{IC}e^i_C-\big[H+\mathcal{O}(e^{-2Ht})+\frac{1}{3}\widehat{n}\big]\widehat{n}\widehat{e}^i_I\\
\notag&+\mathcal{O}(e^{-Ht})\widehat{n}+\mathcal{O}(e^{-Ht})\widehat{n}^2,\\
\label{n.hat.eq}\partial_t\widehat{n}-e_C e_C \widehat{n}+2H\widehat{n}=&-\widehat{\gamma}_{CCD}e_D\widehat{n}-n\widehat{k}_{CD}\widehat{k}_{CD}+\mathcal{O}(e^{-2Ht})\widehat{n}\\
\notag&-\big[2H+\frac{1}{3}n+\mathcal{O}(e^{-2Ht})\big]\widehat{n}^2+n\mathcal{O}(e^{-3Ht}) \widehat{e_0\psi}-n(\widehat{e_0\psi})^2
\end{align}
The scalar field variables $\widehat{e_0\psi},\widehat{e_i\psi}$ satisfy the evolution equations:
\begin{align}
\label{e0psi.hat.eq}\partial_t(\widehat{e_0\psi})+\big[3H+\mathcal{O}(e^{-2Ht})+\widehat{n}\big]\widehat{e_0\psi}=&\,ne_C(\widehat{e_C\psi})-n\widehat{\gamma}_{CCD}\widehat{e_D\psi}+(e_Cn)\widehat{e_C\psi}\\
\notag&-\big[3H+\mathcal{O}(e^{-2Ht})+\widehat{n}\big]\widehat{n}(\widehat{e_0\psi})+\mathcal{O}(e^{-3Ht})\widehat{n}+\mathcal{O}(e^{-3Ht})\widehat{n}^2,\\
\label{eIpsi.hat.eq}\partial_t(\widehat{e_I\psi})+\big[H+\mathcal{O}(e^{-2Ht})+\frac{1}{3}\widehat{n}\big]\widehat{e_I\psi}=&\,n e_I(\widehat{e_0\psi})+(e_In)\widehat{e_0\psi}+\mathcal{O}(e^{-3Ht})e_In+n\widehat{k}_{IC}\widehat{e_C\psi}\\
\notag&-\big[H+\mathcal{O}(e^{-2Ht})+\frac{1}{3}\widehat{n}\big]\widehat{n}(\widehat{e_I\psi}).
\end{align}
Also, the following constraint equation holds:
\begin{align}
e_C \widehat{k}_{CI}=-\frac{2}{3}e_I\widehat{n}+  
\widehat{k}_{ID} \gamma_{CCD}
+
\widehat{k}_{CD} \gamma_{CID}-(e_0\psi)\widehat{e_I\psi}.
\label{mom.hat.const}
\end{align}
\end{lemma}
\begin{proof}
\eqref{k.hat.eq} follows by taking the traceless part of equation \eqref{k.eq}, using \eqref{trk}, and multiplying both sides by $n$. 
We continue with the computations for \eqref{gamma.hat.eq}, using \eqref{gamma.eq}, \eqref{diff.var}, \eqref{red.var.FLRW}: 
\begin{align*}
\partial_t\widehat{\gamma}_{IJB}=&\,ne_B k_{IJ}-ne_Jk_{BI}-nk_{IC}\gamma_{BJC}-nk_{CJ} \gamma_{BIC}+nk_{IC} \gamma_{JBC}
+nk_{BC} \gamma_{JIC}+nk_{IC} \gamma_{CJB}\\
&+(e_B \widehat{n})k_{JI}-(e_J \widehat{n})k_{BI}\\
=&\,ne_B \widehat{k}_{IJ}-ne_J \widehat{k}_{BI}-\frac{1}{3}\delta_{IJ}ne_B\widehat{n}+\frac{1}{3}\delta_{IB}ne_J\widehat{n}
-n\widehat{k}_{IC}\gamma_{BJC}-n\widehat{k}_{CJ} \gamma_{BIC}\\
&+n\widehat{k}_{IC} \gamma_{JBC}
+n\widehat{k}_{BC} \gamma_{JIC}+n\widehat{k}_{IC} \gamma_{CJB}
-n\big[H+\mathcal{O}(e^{-2Ht})+\frac{1}{3}\widehat{n}\big]\gamma_{IJB}
+(e_B \widehat{n})k_{JI}-(e_J \widehat{n})k_{BI},
\end{align*}
as asserted. The computations for \eqref{eIi.hat.eq} are straightforward: 
\begin{align*}
\partial_t \widehat{e}^i_I\overset{\eqref{eIi.eq}}{=}&\,nk_{IC}e^i_C-\partial_t(e^i_I)_{FLRW}\\
=&\,n\widehat{k}_{IC}e^i_C-n\big[H+\mathcal{O}(e^{-2Ht})+\frac{1}{3}\widehat{n}\big]e^i_I-\partial_t(e^i_I)_{FLRW}\\
=&\,n\widehat{k}_{IC}e^i_C-\big[H+\mathcal{O}(e^{-2Ht})+\frac{1}{3}\widehat{n}\big]\widehat{n}\widehat{e}^i_I-\big[H+\mathcal{O}(e^{-2Ht})+\frac{1}{3}\widehat{n}\big]\widehat{e}^i_I+\mathcal{O}(e^{-Ht})\widehat{n}
+\mathcal{O}(e^{-Ht})\widehat{n}^2,
\end{align*}
where in the last equality we made use of \eqref{eIi.eq} for the FLRW variables \eqref{red.var.FLRW}.

For the equation \eqref{n.hat.eq}, we rewrite the RHS of \eqref{n.eq}:
\begin{align*}
&-\gamma_{CCD}e_Dn-nk_{CD}k_{CD}+n\Lambda-n(e_0\psi)^2+\partial_t\mathrm{tr}k_{FLRW}\\
=&-\widehat{\gamma}_{CCD}e_D\widehat{n}-n\widehat{k}_{CD}\widehat{k}_{CD}-\frac{1}{3}n\big[3H+\mathcal{O}(e^{-2Ht})+\widehat{n}\big]^2+n\Lambda\\
&-n(e_0\psi)_{FLRW}^2-n(\widehat{e_0\psi})^2-2n(e_0\psi)_{FLRW}\widehat{e_0\psi}+\partial_t\mathrm{tr}k_{FLRW}\\
=&-\widehat{\gamma}_{CCD}e_D\widehat{n}-n\widehat{k}_{CD}\widehat{k}_{CD}\\
\tag{$H^2=\frac{\Lambda}{3}$}&-2Hn\widehat{n}+\mathcal{O}(e^{-2Ht})\widehat{n}+\mathcal{O}(e^{-2Ht})n\widehat{n}-\frac{1}{3}n\widehat{n}^2-n(\widehat{e_0\psi})^2+n\mathcal{O}(e^{-3Ht})\widehat{e_0\psi}\\
&-(e_0\psi)_{FLRW}^2+\partial_t\mathrm{tr}k_{FLRW}+\mathcal{O}(e^{-2Ht})
\end{align*}
Now the desired equation follows by observing that the sum of the FLRW terms in the last line, including $\mathcal{O}(e^{-2Ht})$, must vanish, since $(\widehat{n},\widehat{k}_{CD},\widehat{\gamma}_{CCD},\widehat{e_0\psi})=(0,0,0,0)$ is a solution to \eqref{n.eq}. 

The computations for \eqref{e0psi.hat.eq}-\eqref{eIpsi.hat.eq} are similar. Plugging \eqref{diff.var}, \eqref{red.var.FLRW} into \eqref{e0psi.eq} we have
\begin{align*}
e_0(\widehat{e_0\psi})+\big[3H+\mathcal{O}(e^{-2Ht})+\widehat{n}\big]\widehat{e_0\psi}=&\,e_C(\widehat{e_C\psi})-\widehat{\gamma}_{CCD}\widehat{e_D\psi}+n^{-1}(e_Cn)\widehat{e_C\psi}\\
\notag&-e_0[(e_0\psi)_{FLRW}]-\big[3H+\mathcal{O}(e^{-2Ht})+\widehat{n}\big](e_0\psi)_{FLRW}
\end{align*}
Multiplying both sides by $n$ gives
\begin{align*}
\partial_t(\widehat{e_0\psi})+\big[3H+\mathcal{O}(e^{-2Ht})+\widehat{n}\big]\widehat{e_0\psi}=&\,ne_C(\widehat{e_C\psi})-n\widehat{\gamma}_{CCD}\widehat{e_D\psi}+(e_Cn)\widehat{e_C\psi}\\
&-\partial_t(e_0\psi)_{FLRW}-\big[3H+\mathcal{O}(e^{-2Ht})+\widehat{n}\big](e_0\psi)_{FLRW}\\
&-\big[3H+\mathcal{O}(e^{-2Ht})+\widehat{n}\big]\widehat{n}(\widehat{e_0\psi})
+\big[3H+\mathcal{O}(e^{-2Ht})+\widehat{n}\big]\mathcal{O}(e^{-3Ht})\widehat{n}
\end{align*}
The equation \eqref{e0psi.hat.eq} follows by observing that the FLRW terms in the second line of the previous equation cancel, since $(\widehat{e_0\psi},\widehat{n},\widehat{\gamma}_{CCD})=(0,0,0)$ is a solution. In turn, \eqref{eIpsi.hat.eq} follows by plugging \eqref{diff.var}, \eqref{red.var.FLRW} into \eqref{eIpsi.eq} for $k_{IC},e_0\psi$, using the formula \eqref{trk} for $\text{tr}k$, and multiplying both sides of the resulting equation by $n$.

Lastly, the constraint equation \eqref{mom.hat.const} follows directly from \eqref{momconst} by plugging in \eqref{diff.var}, \eqref{red.var.FLRW}, and using the anti-symmetry of $\gamma_{CCI}=-\gamma_{CIC}$:
\begin{align*}
e_C\widehat{k}_{CI}=&\,e_Ck_{CI}+\frac{1}{3}e_In=-\frac{2}{3}e_I\widehat{n}+k_{ID}\gamma_{CCD}+k_{CD}\gamma_{CID}-(e_0\psi) e_I\psi\\
=&-\frac{2}{3}e_I\widehat{n}+\widehat{k}_{ID}\gamma_{CCD}+\widehat{k}_{CD}\gamma_{CID}+\frac{1}{3}\text{tr}k\gamma_{CCI}+\frac{1}{3}\text{tr}k\gamma_{CIC}-(e_0\psi) e_I\psi\\
=&-\frac{2}{3}e_I\widehat{n}+\widehat{k}_{ID}\gamma_{CCD}+\widehat{k}_{CD}\gamma_{CID}-(e_0\psi)\widehat{e_I\psi}
\end{align*}
This completes the proof of the lemma.
\end{proof}
\section{Norms, total energy, and bootstrap assumptions}\label{sec:norms.boots}

We use standard $H^M$-based norms to control the perturbed solution. Our bootstrap assumptions are expressed as a bound on the total energy of the system \eqref{k.hat.eq}-\eqref{eIpsi.hat.eq}, defined in \eqref{tot.en}.

\subsection{Norms of the unknown variables}\label{subsec:norms}

Define the norm $\|v\|_{L^2(\Sigma_t)}$ for a scalar function by
\begin{align}\label{L2}
\|v\|_{L^2(\Sigma_t)}^2:=\int_{\Sigma_t}v^2(t,x)dx,
\end{align}
where $dx=dx^1dx^2dx^3$. Define also the corresponding $H^M(\Sigma_t),C^{M}(\Sigma_t)$ norms: 
\begin{align}\label{HM}
\|v\|_{H^M(\Sigma_t)}^2=\sum_{|\iota|\leq M}\|\partial^\iota v\|_{L^2(\Sigma_t)},\qquad\|v\|_{C^M(\Sigma_t)}=\sum_{|\iota|\leq M}\|\partial^\iota v\|_{C^0(\Sigma_t)},
\end{align}
where is $\iota$ is a spatial multi-index and $\partial^\iota$ is the operator that acts on $v$ with repeated differentiation with respect to the holonomic vector fields $\partial_1,\partial_2,\partial_3$. 

For the indexed variables that we work with, their corresponding $H^M(\Sigma_t),C^{M}(\Sigma_t)$ norms are simply the sum of the norms of their components:
\begin{align}\label{HM.hat}
\begin{split}
\|\widehat{k}\|_{H^M(\Sigma_t)}^2=\sum_{I,J=1}^3\|\widehat{k}_{IJ}\|_{H^M(\Sigma_t)}^2,&\qquad\|\widehat{k}\|_{C^{M}(\Sigma_t)}^2=\sum_{I,J=1}^3\|\widehat{k}_{IJ}\|_{C^{M}(\Sigma_t)}^2\\
\|\widehat{\gamma}\|_{H^M(\Sigma_t)}^2=\sum_{I,J,B=1}^3\|\widehat{\gamma}_{IJB}\|_{H^M(\Sigma_t)}^2,&\qquad\|\widehat{\gamma}\|_{C^{M}(\Sigma_t)}^2=\sum_{I,J,B=1}^3\|\widehat{\gamma}_{IJB}\|_{C^{M}(\Sigma_t)}^2\\
\|\widehat{e}\|_{H^M(\Sigma_t)}^2=\sum_{i,I=1}^3\|\widehat{e}^i_I\|_{H^M(\Sigma_t)}^2,&\qquad\|\widehat{e}\|_{C^{M}(\Sigma_t)}^2=\sum_{i,I=1}^3\|\widehat{e}^i_I\|_{C^{M}(\Sigma_t)}^2\\
\|\widehat{e\psi}\|_{H^M(\Sigma_t)}^2=\sum_{\mu=0}^3\|\widehat{e_\mu\psi}\|_{H^M(\Sigma_t)}^2,&\qquad\|\widehat{e\psi}\|_{C^{M}(\Sigma_t)}^2=\sum_{\mu=0}^3\|\widehat{e_\mu\psi}\|_{C^{M}(\Sigma_t)}^2
\end{split}
\end{align}
\subsection{Total energy}\label{subsec:tot.en}

Let $N\ge 4$. The total energy of our reduced system reads:
\begin{align}\label{tot.en}
\mathcal{E}(t)=e^{2Ht}\big\{\|\widehat{k}\|_{H^N(\Sigma_t)}^2
+\|\widehat{\gamma}\|_{H^N(\Sigma_t)}^2+\|\widehat{e}\|_{H^N(\Sigma_t)}^2+e^{Ht}\|\widehat{n}\|_{H^N(\Sigma_t)}^2
+\|\widehat{e\psi}\|_{H^N(\Sigma_t)}^2\big\}
\end{align}
for all $t\ge0$.

\subsection{Bootstrap assumptions}\label{subsec:Boots}

Our bootstrap assumptions are that there exists a bootstrap time $T_{Boot}\in(0,+\infty)$ such that it holds
\begin{align}\label{Boots}
\mathcal{E}(t)\leq \varepsilon^2, \qquad\forall t\in[0,T_{Boot}),
\end{align}
for a sufficiently small constant $\varepsilon>0$ to be determined below. Such a $T_{Boot}>0$ obviously exists by Cauchy stability, assuming that the perturbed initial data on $\Sigma_0$ are sufficiently close to the FLRW initial data in the above energy space.

\section{Global existence estimates}\label{sec:main.est}

In this section we derive the exponential decay of the variables $\widehat{k}_{IJ},\widehat{\gamma}_{IJB},\widehat{e}^i_I,\widehat{n},\widehat{e_0\psi},\widehat{e_I\psi}$, in a manner that yields a strict 
improvement of our bootstrap assumptions \eqref{Boots}, which in turn, by a standard continuation argument, implies that $T_{Boot}=+\infty$, ie. the perturbed solution exists for all time and it satisfies the estimate \eqref{Boots} for all $t\in[0,+\infty)$, see Proposition \ref{prop:cont.arg}.

\subsection{Basic estimates and identities}

We will frequently use the standard Sobolev inequality in $(\mathbb{T}^3,dx)$:
\begin{align}\label{Sob}
\|v\|_{C^0(\Sigma_t)}\leq C\|v\|_{H^2(\Sigma_t)},
\end{align}
where $C>0$ depends on $\mathbb{T}^3$.
Already, \eqref{Boots} and \eqref{Sob} imply the bounds 
\begin{align}\label{CN-2.est}
\begin{split}
\|\widehat{k}\|_{C^{N-2}(\Sigma_t)}
+\|\widehat{\gamma}\|_{C^{N-2}(\Sigma_t)}+\|\widehat{e}\|_{C^{N-2}(\Sigma_t)}+\|\widehat{e\psi}\|_{C^{N-2}(\Sigma_t)}\leq&\,C\varepsilon e^{-Ht},\\
\|\widehat{n}\|_{C^{N-2}(\Sigma_t)}\leq&\,C\varepsilon e^{-\frac{3}{2}Ht},
\end{split}
\end{align}
for all $t\in[0,T_{Boot})$.
When we commute the equations \eqref{k.hat.eq}-\eqref{mom.hat.const} with $\partial^\iota$, we will use the commutator relation:
\begin{align}\label{comm.eI}
[\partial^\iota,e_I]v=\sum_{\iota_1\cup\iota_2=\iota,\,|\iota_2|<|\iota|}(\partial^{\iota_1}e^a_I)\partial^{\iota_2}\partial_av.
\end{align}
We also have the integration by parts formula relative to $e_I$:
\begin{align}\label{IBP}
\int_{\Sigma_t}v_1 (e_Iv_2) dx=-\int_{\Sigma_t}[(e_Iv_1)v_2+(\partial_ie^i_I)v_1v_2]dx.
\end{align}
\subsection{Estimates for $\widehat{n}$}\label{subsec:est.n}

First, we compute the commuted version of the equation \eqref{n.hat.eq}, with $\partial^\iota$, $|\iota|\leq N$, using \eqref{comm.eI}:
\begin{align}\label{high.n.hat.eq}
\notag&\partial_t \partial^\iota\widehat{n}-e_C\partial^\iota e_C\widehat{n}+2H\partial^\iota \widehat{n}\\
=&-\widehat{\gamma}_{CCD}e_D\partial^\iota\widehat{n}
-\sum_{\iota_1\cup\iota_2\cup\iota_3=\iota,\,|\iota_3|<|\iota|}\partial^{\iota_1}\widehat{\gamma}_{CCD}(\partial^{\iota_2}e^a_D)\partial^{\iota_3}\partial_a\widehat{n}
+\sum_{\iota_1\cup\iota_2=\iota,\,|\iota_2|<|\iota|}(\partial^{\iota_1}e^a_C)\partial_a\partial^{\iota_2}e_C\widehat{n}\\
\notag&-\partial^\iota\big\{n\widehat{k}_{CD}\widehat{k}_{CD}+\mathcal{O}(e^{-2Ht})\widehat{n}+\big[2H+\frac{1}{3}n+\mathcal{O}(e^{-2Ht})\big]\widehat{n}^2+n\mathcal{O}(e^{-3Ht}) \widehat{e_0\psi}+n(\widehat{e_0\psi})^2\big\}
\end{align}
\begin{proposition}\label{prop:n.est}
Let $N\ge4$ and assume the bootstrap assumptions \eqref{Boots} are valid.  Then the following estimate holds:
\begin{align}\label{n.hat.est}
\frac{1}{2}\partial_t[e^{3Ht}\|\widehat{n}\|^2_{H^N(\Sigma_t)}]
+\frac{1}{2}He^{3Ht}\|\widehat{n}\|^2_{H^N(\Sigma_t)}+(1-\eta)\sum_{C=1}^3e^{3Ht}\|e_C\widehat{n}\|^2_{H^N(\Sigma_t)}\leq \frac{C}{\eta}e^{-\frac{1}{2}Ht}\mathcal{E}(t),
\end{align}
for all $t\in[0,T_{Boot})$, and some sufficiently small constant $\eta>0$ to be chosen later.
\end{proposition}
\begin{proof}
Differentiate the weighted $H^N(\Sigma_t)$ norm of $\widehat{n}$ in $\partial_t$ and plug in \eqref{high.n.hat.eq}:
\begin{align}\label{n.hat.est2}
\notag&\frac{1}{2}\partial_t[e^{3Ht}\|\widehat{n}\|^2_{H^N(\Sigma_t)}]\\
\notag=&\,\frac{3}{2}He^{3Ht}\|\widehat{n}\|^2_{H^N(\Sigma_t)}+\sum_{|\iota|\leq N}\int_{\Sigma_t}e^{3Ht}(\partial^\iota\widehat{n})\partial^\iota\partial_t\widehat{n}
dx\\
=&-\frac{1}{2}He^{3Ht}\|\widehat{n}\|^2_{H^N(\Sigma_t)}
+\sum_{|\iota|\leq N}\int_{\Sigma_t}e^{3Ht}\partial^\iota\widehat{n}\bigg[e_C\partial^\iota e_C\widehat{n}-\widehat{\gamma}_{CCD}e_D\partial^\iota\widehat{n}+\sum_{\iota_1\cup\iota_2=\iota,\,|\iota_2|<|\iota|}(\partial^{\iota_1}e^a_C)\partial_a\partial^{\iota_2}e_C\widehat{n}\\
\notag&-\partial^\iota\big\{n\widehat{k}_{CD}\widehat{k}_{CD}+\mathcal{O}(e^{-2Ht})\widehat{n}+\big[2H+\frac{1}{3}n+\mathcal{O}(e^{-2Ht})\big]\widehat{n}^2+n\mathcal{O}(e^{-3Ht}) \widehat{e_0\psi}+n(\widehat{e_0\psi})^2\big\}\\
\notag&-\sum_{\iota_1\cup\iota_2\cup\iota_3=\iota,\,|\iota_3|<|\iota|}\partial^{\iota_1}\widehat{\gamma}_{CCD}(\partial^{\iota_2}e^a_D)\partial^{\iota_3}\partial_a\widehat{n}\bigg]dx
\end{align}
Notice that every term in the last two lines contains at most one factor with more than $N-2$ spatial derivatives, since $N\ge4$. Hence, they can be directly estimated using Cauchy-Schwarz and the $C^{N-2}(\Sigma_t)$ bounds \eqref{CN-2.est}, to deduce the inequality: 
\begin{align}\label{n.hat.est3}
\notag\frac{1}{2}\partial_t[e^{3Ht}\|\widehat{n}\|^2_{H^N(\Sigma_t)}]
+\frac{1}{2}He^{3Ht}\|\widehat{n}\|^2_{H^N(\Sigma_t)}
\leq&\, Ce^{-\frac{1}{2}Ht}\mathcal{E}(t)\\
&+\sum_{|\iota|\leq N}\bigg[\int_{\Sigma_t}e^{3Ht}\partial^\iota\widehat{n}(e_C\partial^\iota e_C\widehat{n}-\widehat{\gamma}_{CCD}e_D\partial^\iota\widehat{n})dx\\
\notag&+\int_{\Sigma_t}e^{3Ht}\partial^\iota\widehat{n}\sum_{\iota_1\cup\iota_2=\iota,\,|\iota_2|<|\iota|}(\partial^{\iota_1}e^a_C)\partial_a\partial^{\iota_2}e_C\widehat{n}dx\bigg]
\end{align}
We estimate the last term using \eqref{Sob}, the bootstrap assumptions \eqref{Boots}, and Young's inequality:
\begin{align}\label{n.hat.est4}
\sum_{|\iota|\leq N}\int_{\Sigma_t}e^{3Ht}\partial^\iota\widehat{n}\sum_{\iota_1\cup\iota_2=\iota,\,|\iota_2|<|\iota|}(\partial^{\iota_1}e^a_C)\partial_a\partial^{\iota_2}e_C\widehat{n}dx
\leq\frac{C}{\eta}e^{-Ht}\mathcal{E}(t)+\eta\sum_{C=1}^3e^{3Ht}\|e_C\widehat{n}\|_{H^N(\Sigma_t)}^2,
\end{align}
for a constant $\eta>0$ of our choice.
The remaining two terms in the second line of \eqref{n.hat.est3} are treated by integrating by parts, using \eqref{IBP}:
\begin{align}\label{n.hat.est5}
\notag&\sum_{|\iota|\leq N}\int_{\Sigma_t}e^{3Ht}\partial^\iota\widehat{n}(e_C\partial^\iota e_C\widehat{n}-\widehat{\gamma}_{CCD}e_D\partial^\iota\widehat{n})dx\\
\notag=&-\sum_{|\iota|\leq N}\int_{\Sigma_t}e^{3Ht}\bigg[e_C\partial^\iota\widehat{n}(\partial^\iota e_C\widehat{n}-\widehat{\gamma}_{DDC}\partial^\iota\widehat{n})+(\partial_ae_C^a)\partial^\iota\widehat{n}\partial^\iota e_C\widehat{n}-\partial_a(e_C^a\widehat{\gamma}_{DDC})\partial^\iota\widehat{n}\bigg]dx\\
=&-\sum_{|\iota|\leq N}\int_{\Sigma_t}e^{3Ht}\bigg[\partial^\iota e_C\widehat{n}(\partial^\iota e_C\widehat{n}-\widehat{\gamma}_{DDC}\partial^\iota\widehat{n})-\sum_{\iota_1\cup\iota_2=\iota,\,|\iota_2|<|\iota|}(\partial^{\iota_1} e_C^a)\partial_a\partial^{\iota_2}\widehat{n}(\partial^\iota e_C\widehat{n}-\widehat{\gamma}_{DDC}\partial^\iota\widehat{n})\\
\notag&+(\partial_ae_C^a)\partial^\iota\widehat{n}\partial^\iota e_C\widehat{n}-\partial_a(e_C^a\widehat{\gamma}_{DDC})\partial^\iota\widehat{n}\bigg]dx\\
\notag\leq&\,(\eta-1)\sum_{C=1}^3e^{3Ht}\|e_C\widehat{n}\|_{H^N(\Sigma_t)}^2+\frac{C}{\eta}e^{-Ht}\mathcal{E}(t),
\end{align}
where we employed again Cauchy-Schwarz, the $C^{N-2}(\Sigma_t)$ bounds \eqref{CN-2.est}, and Young's inequality.
Combining \eqref{n.hat.est3}-\eqref{n.hat.est5} we conclude \eqref{n.hat.est}, for a different constant $\eta>0$ than the ones in \eqref{n.hat.est4}, \eqref{n.hat.est5}. 
\end{proof}
\subsection{Estimates for $\widehat{e}^i_I$}

Commuting \eqref{eIi.hat.eq} with $\partial^\iota$, $|\iota|\leq N$, we have:
\begin{align}\label{high.eIi.hat.eq}
\partial_t \partial^\iota\widehat{e}_I^i+\big[H+\mathcal{O}(e^{-2Ht})+\frac{1}{3}\widehat{n}\big]\partial^\iota\widehat{e}^i_I=&\,\partial^\iota\big\{n\widehat{k}_{IC}e^i_C-\big[H+\mathcal{O}(e^{-2Ht})+\frac{1}{3}\widehat{n}\big]\widehat{n}\widehat{e}^i_I\\
\notag&+\mathcal{O}(e^{-Ht})\widehat{n}+\mathcal{O}(e^{-Ht})\widehat{n}^2\big\}-\sum_{\iota_1\cup\iota_2=\iota,\,|\iota_2|<|\iota|}\frac{1}{3}\partial^{\iota_1}\widehat{n}\partial^{\iota_2}\widehat{e}^i_I
\end{align}

\begin{proposition}\label{prop:eIi.est}
Let $N\ge4$ and assume the bootstrap assumptions \eqref{Boots} are valid.  Then the following estimate holds:
\begin{align}\label{eIi.hat.est}
\partial_t[e^{2Ht}\|\widehat{e}\|^2_{H^N(\Sigma_t)}]\leq Ce^{-Ht}\mathcal{E}(t),
\end{align}
for all $t\in[0,T_{Boot})$.
\end{proposition}
\begin{proof}
After differentiating, in $\partial_t$, the $H^N(\Sigma_t)$ norm of $\widehat{e}$ and plugging in \eqref{high.eIi.hat.eq}, the desired estimate follows by a straightforward application of Cauchy-Schwarz and the $C^{N-2}(\Sigma_t)$ bounds \eqref{CN-2.est}, since each term in the RHS \eqref{high.eIi.hat.eq} has at most one factor with more than $N-2$ spatial derivatives:
\begin{align}
\label{eIi.hat.est}
\notag\frac{1}{2}\partial_t[e^{2Ht}\|\widehat{e}\|^2_{H^N(\Sigma_t)}]=&\,He^{2Ht}\|\widehat{e}\|^2_{H^N(\Sigma_t)}-\sum_{a,I=1}^3\sum_{|\iota|\leq N}\int_{\Sigma_t}\big[H+\mathcal{O}(e^{-2Ht})+\frac{1}{3}\widehat{n}\big]e^{2Ht}(\partial^\iota\widehat{e}^a_I)^2\\
&+\sum_{a,I=1}^3\sum_{|\iota|\leq N}\int_{\Sigma_t}e^{2Ht}\partial^\iota\widehat{e}^a_I
\bigg[\partial^\iota\big\{n\widehat{k}_{IC}e^i_C-\big[H+\mathcal{O}(e^{-2Ht})+\frac{1}{3}\widehat{n}\big]\widehat{n}\widehat{e}^i_I\\
\notag&+\mathcal{O}(e^{-Ht})\widehat{n}+\mathcal{O}(e^{-Ht})\widehat{n}^2\big\}-\sum_{\iota_1\cup\iota_2=\iota,\,|\iota_2|<|\iota|}\frac{1}{3}\partial^{\iota_1}\widehat{n}\partial^{\iota_2}\widehat{e}^i_I
\bigg]dx\\
\leq&\,Ce^{-Ht}\mathcal{E}(t)\notag,
\end{align}
where we note that the first two terms in the first line of the previous RHS cancel.
\end{proof}
\subsection{Estimates for $\widehat{k}_{IJ},\widehat{\gamma}_{IJB}$}

To derive higher order energy estimates for $\widehat{k},\widehat{\gamma}$, apart from the evolution equations \eqref{k.hat.eq}-\eqref{gamma.hat.eq}, we will also make use of the constraint equation \eqref{mom.hat.const}. The differentiated versions of these, with $\partial^\iota$, $|\iota|\leq N$, read:
\begin{align}
\label{high.k.hat.eq}\partial_t \partial^\iota\widehat{k}_{IJ}+\big[3H+\mathcal{O}(e^{-2Ht})+\widehat{n}\big]\partial^\iota\widehat{k}_{IJ}
=&-e_I \partial^\iota e_J\widehat{n}+\frac{1}{3}\delta_{IJ}e_C \partial^\iota e_C \widehat{n}\\
\notag&+n(e_C\partial^\iota \gamma_{IJC}-e_I\partial^\iota\gamma_{CJC}-\frac{2}{3}\delta_{IJ}e_C\partial^\iota \gamma_{DDC})+\mathfrak{K}^\iota_{IJ},\\
\label{high.gamma.hat.eq}\partial_t \partial^\iota\widehat{\gamma}_{IJB}+\big[H+\mathcal{O}(e^{-2Ht})+\frac{1}{3}\widehat{n}\big]\partial^\iota\widehat{\gamma}_{IJB}
=&\,ne_B\partial^\iota \widehat{k}_{IJ}-ne_J\partial^\iota \widehat{k}_{BI}
+\mathfrak{G}_{IJB}^\iota,
\end{align}
and
\begin{align}\label{high.mom.hat.const}
e_C\partial^\iota\widehat{k}_{CI}=\partial^\iota\big\{\widehat{k}_{ID}\gamma_{CCD}+\widehat{k}_{CD}\gamma_{CID}-\frac{2}{3} e_I\widehat{n}-(e_0\psi)\widehat{e_I\psi}\big\}-\sum_{\iota_1\cup\iota_2=\iota\,|\iota_2|<|\iota|}(\partial^{\iota_1}e_C^a)\partial_a\partial^{\iota_2}\widehat{k}_{CI},
\end{align}
where 
\begin{align}
\label{frakK}\mathfrak{K}^\iota_{IJ}=&-\sum_{\iota_1\cup\iota_2=\iota\,|\iota_2|<|\iota|}\big\{(\partial^{\iota_1}e_I^a)\partial_a \partial^{\iota_2} e_J\widehat{n}-\frac{1}{3}\delta_{IJ}(\partial^{\iota_1}e_C^a)\partial_a \partial^{\iota_2} e_C \widehat{n}\big\}\\
\notag&+\sum_{\iota_1\cup\iota_2\cup\iota_3=\iota\,|\iota_3|<|\iota|}(\partial^{\iota_1}n)\big\{(\partial^{\iota_2}e_C^a)\partial_a\partial^{\iota_3} \gamma_{IJC}-(\partial^{\iota_2}e_I^a)\partial_a\partial^{\iota_3}\gamma_{CJC}-\frac{2}{3}\delta_{IJ}(\partial^{\iota_2}e_C^a)\partial_a\partial^{\iota_3} \gamma_{DDC}\big\}\\
\notag&+\partial^\iota\big\{\gamma_{IJC} e_C n-n\gamma_{CID}\gamma_{DJC}-n\gamma_{IJD}\gamma_{CCD}-\frac{1}{3}\delta_{IJ}(\gamma_{CCD} e_D n
-n\gamma_{CDE}\gamma_{EDC}-n\gamma_{CCD}\gamma_{EED})\\
\notag&-ne_I\psi e_J\psi+\frac{1}{3}\delta_{IJ}ne_C\psi e_C\psi-\big[3H+\mathcal{O}(e^{-2Ht})+\widehat{n}\big]\widehat{n}\widehat{k}_{IJ}\big\}-\sum_{\iota_1\cup\iota_2=\iota\,|\iota_2|<|\iota|}(\partial^{\iota_1}\widehat{n})\partial^{\iota_2}\widehat{k}_{IJ},\\
\label{frakG}\mathfrak{G}_{IJB}^\iota=&\sum_{\iota_1\cup\iota_2\cup\iota_3=\iota\,|\iota_3|<|\iota|}\big\{\partial^{\iota_1}n(\partial^{\iota_2}e_B^a)\partial_a\partial^{\iota_3} \widehat{k}_{IJ}-\partial^{\iota_1}n(\partial^{\iota_2}e_J^a)\partial_a\partial^{\iota_3} \widehat{k}_{BI}\big\}\\
\notag&+\partial^\iota\bigg[\frac{1}{3}\delta_{IB}ne_J\widehat{n}-\frac{1}{3}\delta_{IJ}ne_B\widehat{n}
-n\widehat{k}_{IC}\gamma_{BJC}
-n\widehat{k}_{CJ} \gamma_{BIC}+n\widehat{k}_{IC} \gamma_{JBC}
+n\widehat{k}_{BC} \gamma_{JIC}+n\widehat{k}_{IC} \gamma_{CJB}\\
\notag&+(e_B \widehat{n})k_{JI}-(e_J \widehat{n})k_{BI}-\big[H+\mathcal{O}(e^{-2Ht})+\frac{1}{3}\widehat{n}\big]\widehat{n}\widehat{\gamma}_{IJB}\bigg]-\sum_{\iota_1\cup\iota_2=\iota\,|\iota_2|<|\iota|}\frac{1}{3}(\partial^{\iota_1}\widehat{n})\partial^{\iota_2}\widehat{\gamma}_{IJB}
\end{align}
\begin{lemma}\label{lem:high.en.id}
Let $\iota$ be a spatial multi-index with $|\iota|\leq N$. Then the following high order energy identity holds:
\begin{align}\label{high.en.id}
\notag&\frac{1}{2}\partial_t[e^{2Ht}(\partial^\iota\widehat{k}_{IJ})\partial^\iota\widehat{k}_{IJ}]+\frac{1}{4}\partial_t[e^{2Ht}(\partial^\iota\widehat{\gamma}_{IJB})\partial^\iota\widehat{\gamma}_{IJB}]\\
\notag=&-\big[2H+\mathcal{O}(e^{-2Ht})+\widehat{n}\big]e^{2Ht}(\partial^\iota\widehat{k}_{IJ})\partial^\iota\widehat{k}_{IJ}+\big[\mathcal{O}(e^{-2Ht})-\frac{1}{6}\widehat{n}\big]e^{2Ht}(\partial^\iota\widehat{\gamma}_{IJB})\partial^\iota\widehat{\gamma}_{IJB}\\
&+e^{2Ht}\big\{ne_C(\partial^\iota \widehat{k}_{IJ}\partial^\iota \widehat{\gamma}_{IJC})-e_I(\partial^\iota \widehat{k}_{IJ}\partial^\iota e_J\widehat{n})-ne_I(\partial^\iota \widehat{k}_{IJ}\partial^\iota \widehat{\gamma}_{CJC})\big\}
+e^{2Ht}\partial^\iota \widehat{k}_{IJ}\mathfrak{K}_{IJ}^\iota\\
\notag&+\frac{1}{2}e^{2Ht}\partial^\iota\widehat{\gamma}_{IJB}\mathfrak{G}_{IJB}^\iota
+e^{2Ht}(\partial^\iota e_J\widehat{n}+n\partial^\iota \widehat{\gamma}_{CJC})\bigg[\partial^\iota\big\{\widehat{k}_{JD}\gamma_{CCD}+\widehat{k}_{CD}\gamma_{CJD}\\
\notag&-\frac{2}{3} e_J\widehat{n}-(e_0\psi)\widehat{e_J\psi}\big\}-\sum_{\iota_1\cup\iota_2=\iota\,|\iota_2|<|\iota|}(\partial^{\iota_1}e_C^a)\partial_a\partial^{\iota_2}\widehat{k}_{CJ}\bigg]
\end{align}
\end{lemma}
\begin{proof}
It is a straightforward computation using \eqref{high.k.hat.eq}-\eqref{high.gamma.hat.eq}, the fact that $\widehat{k}_{IJ}$ is traceless, the anti-symmetry $\widehat{\gamma}_{IJB}=-\widehat{\gamma}_{IBJ}$, and the differentiated constraint \eqref{high.mom.hat.const} to replace the factor $e_I\partial^\iota\widehat{k}_{IJ}$ in two terms, giving the last bracket in \eqref{high.en.id}.
\end{proof}
\begin{lemma}\label{lem:k.gamma.error.est}
Let $N\ge4$ and assume the bootstrap assumptions \eqref{Boots} are valid. Recall the definitions \eqref{frakK}-\eqref{frakG} of $\mathfrak{K}^\iota_I,\mathfrak{G}^\iota_{IJB}$. Then the following estimates hold:
\begin{align}\label{k.gamma.error.est}
\sum_{I,J=1}^3e^{Ht}\|\mathfrak{K}^\iota_{IJ}\|_{L^2(\Sigma_t)}+\sum_{I,J,B=1}^3e^{Ht}\|\mathfrak{G}^\iota_{IJB}\|_{L^2(\Sigma_t)}\leq&\,Ce^{-\frac{1}{2}Ht}\sum_{J=1}^3e^{\frac{3}{2}Ht}\|e_J\widehat{n}\|_{H^N(\Sigma_t)}+Ce^{-Ht}\sqrt{\mathcal{E}(t)}
\end{align}
and
\begin{align}\label{k.gamma.error.est2}
\notag&\int_{\Sigma_t}e^{2Ht}(\partial^\iota e_J\widehat{n}+n\partial^\iota \widehat{\gamma}_{CJC})\bigg[\partial^\iota\big\{\widehat{k}_{JD}\gamma_{CCD}+\widehat{k}_{CD}\gamma_{CJD}
-\frac{2}{3} e_J\widehat{n}-(e_0\psi)\widehat{e_J\psi}\big\}\\
&-\sum_{\iota_1\cup\iota_2=\iota\,|\iota_2|<|\iota|}(\partial^{\iota_1}e_C^a)\partial_a\partial^{\iota_2}\widehat{k}_{CJ}\bigg]dx\\
\notag\leq&\,\eta\sum_{J=1}^3e^{3Ht}\|e_J\eta\|_{H^N(\Sigma_t)}^2+\frac{C}{\eta}e^{-Ht}\mathcal{E}(t)
\end{align}
for all $t\in[0,T_{Boot})$, $|\iota|\leq N$, and some sufficiently small constant $\eta>0$ to be chosen later.
\end{lemma}
\begin{proof}
Both estimates follow by Cauchy-Schwarz, the Sobolev inequality \eqref{Sob}, the bootstrap assumptions \eqref{Boots}, the $C^{N-2}(\Sigma_t)$ bounds \eqref{CN-2.est}, and Young's inequality. For \eqref{k.gamma.error.est}, the less decaying coefficient $e^{-\frac{1}{2}Ht}$ comes from the terms $\partial^\iota[(e_B\widehat{n})k_{IJ}-(e_J\widehat{n})k_{BI}]$ in \eqref{frakG}, when all derivatives $\partial^\iota$ act on the factors $e_B\widehat{n},e_J\widehat{n}$. For \eqref{k.gamma.error.est2}, we also need to observe that the term $-\frac{2}{3}e^{2Ht}(\partial^\iota e_J\widehat{n})\partial^\iota e_J\widehat{n}$ is negative and we can therefore discard it.
\end{proof}
\begin{proposition}\label{prop:k.gamma.est}
Let $N\ge4$ and assume the bootstrap assumptions \eqref{Boots} are valid.  Then the following estimate holds:
\begin{align}\label{high.k.gamma.est}
\begin{split}
&\frac{1}{2}\partial_t[e^{2Ht}\|\widehat{k}\|^2_{H^N(\Sigma_t)}]+\frac{1}{4}\partial_t[e^{2Ht}\|\widehat{\gamma}\|^2_{H^N(\Sigma_t)}]+2He^{2Ht}\|\widehat{k}\|^2_{H^N(\Sigma_t)}\\
\leq&\, \eta \sum_{J=1}^3e^{3Ht}\|e_J\widehat{n}\|^2_{H^N(\Sigma_t)}+\frac{C}{\eta}e^{-\frac{1}{2}Ht}\mathcal{E}(t),
\end{split}
\end{align}
for all $t\in[0,T_{Boot})$ and some sufficiently small constant $\eta>0$ to be chosen later.
\end{proposition}
\begin{proof}
Integrating \eqref{high.en.id} in $\Sigma_t$, summing over $|\iota|\leq N$, using the $C^{N-2}(\Sigma_t)$ bounds \eqref{CN-2.est}, the estimates in Lemma \ref{lem:k.gamma.error.est} and Young's inequality, we obtain the differential inequality
\begin{align}\label{high.k.gamma.est2}
\notag&\frac{1}{2}e^{2Ht}\|\widehat{k}\|^2_{H^N(\Sigma_t)}+\frac{1}{4}e^{2Ht}\|\widehat{\gamma}\|^2_{H^N(\Sigma_t)}+2He^{2Ht}\|\widehat{k}\|^2_{H^N(\Sigma_t)}\\
\leq&\, \eta \sum_{J=1}^3e^{3Ht}\|e_J\widehat{n}\|^2_{H^N(\Sigma_t)}+\frac{C}{\eta}e^{-\frac{1}{2}Ht}\mathcal{E}(t)\\
\notag&+\sum_{|\iota|\leq N}\int_{\Sigma_t}e^{2Ht}\big\{ne_C(\partial^\iota \widehat{k}_{IJ}\partial^\iota \widehat{\gamma}_{IJC})-e_I(\partial^\iota \widehat{k}_{IJ}\partial^\iota e_J\widehat{n})-ne_I(\partial^\iota \widehat{k}_{IJ}\partial^\iota \widehat{\gamma}_{CJC})\big\}dx
\end{align}
To control the terms in the last line, we integrate by parts in $e_C,e_I$ using \eqref{IBP}:
\begin{align}\label{high.k.gamma.est3}
\notag&\sum_{|\iota|\leq N}\int_{\Sigma_t}e^{2Ht}\big\{ne_C(\partial^\iota \widehat{k}_{IJ}\partial^\iota \widehat{\gamma}_{IJC})-e_I(\partial^\iota \widehat{k}_{IJ}\partial^\iota e_J\widehat{n})-ne_I(\partial^\iota \widehat{k}_{IJ}\partial^\iota \widehat{\gamma}_{CJC})\big\}dx\\
=&-\sum_{|\iota|\leq N}\int_{\Sigma_t}e^{2Ht}\big\{[\partial_a(ne_C^a)](\partial^\iota \widehat{k}_{IJ}\partial^\iota \widehat{\gamma}_{IJC})-(\partial_ae_I^a)(\partial^\iota \widehat{k}_{IJ}\partial^\iota e_J\widehat{n})-[\partial_a(ne_I^a)](\partial^\iota \widehat{k}_{IJ}\partial^\iota \widehat{\gamma}_{CJC})\big\}dx\\
\tag{by \eqref{CN-2.est}, Cauchy-Schwarz, and Young's inequality}\leq&\,\eta \sum_{J=1}^3e^{3Ht}\|e_J\widehat{n}\|^2_{H^N(\Sigma_t)}+\frac{C}{\eta}e^{-\frac{1}{2}Ht}\mathcal{E}(t),
\end{align}
The combination of \eqref{high.k.gamma.est2}-\eqref{high.k.gamma.est3} yields \eqref{high.gamma.hat.eq} (for a different $\eta$). 
\end{proof}
\subsection{Estimates for $\widehat{e_0\psi},\widehat{e_I\psi}$}\label{subsec:est.psi}

The estimates for the scalar field are carried out by treating \eqref{e0psi.hat.eq}-\eqref{eIpsi.hat.eq} as a first order symmetric system for $\widehat{e_0\psi},\widehat{e_I\psi}$. Commuting the latter equations with $\partial^\iota$, $|\iota|\leq N$, we obtain:
\begin{align}
\label{high.e0psi.hat.eq}\partial_t\partial^\iota(\widehat{e_0\psi})+\big[3H+\mathcal{O}(e^{-2Ht})+\widehat{n}\big]\partial^\iota(\widehat{e_0\psi})
=&\,ne_C\partial^\iota(\widehat{e_C\psi})+\mathfrak{F}_0^\iota,\\
\label{high.eIpsi.hat.eq}\partial_t\partial^\iota(\widehat{e_I\psi})+\big[H+\mathcal{O}(e^{-2Ht})+\frac{1}{3}\widehat{n}\big]\partial^\iota(\widehat{e_I\psi})=&\,n e_I\partial^\iota(\widehat{e_0\psi})+\mathfrak{F}_I^\iota,
\end{align}
where
\begin{align}
\label{frak.F.0}\mathfrak{F}_0^\iota=&\,\partial^\iota\big\{(e_Cn)\widehat{e_C\psi}-n\widehat{\gamma}_{CCD}\widehat{e_D\psi}
-\big[3H+\mathcal{O}(e^{-2Ht})+\widehat{n}\big]\widehat{n}(\widehat{e_0\psi})+\mathcal{O}(e^{-3Ht})\widehat{n}\\
\notag&+\mathcal{O}(e^{-3Ht})\widehat{n}^2\big\}
+\sum_{\iota_1\cup\iota_2\cup\iota_3=\iota,\,|\iota_3|<|\iota|}\partial^{\iota_1}n(\partial^{\iota_2}e_C^a)\partial_a\partial^{\iota_3}(\widehat{e_C\psi})-\sum_{\iota_1\cup\iota_2=\iota,\,|\iota_2|<|\iota|}\partial^{\iota_1}\widehat{n}\partial^{\iota_2}(\widehat{e_0\psi}),\\
\label{frak.F.2}\mathfrak{F}_I^\iota=&\,\partial^\iota\big\{(e_In)\widehat{e_0\psi}+\mathcal{O}(e^{-3Ht})e_In+n\widehat{k}_{IC}\widehat{e_C\psi}-\big[H+\mathcal{O}(e^{-2Ht})+\frac{1}{3}\widehat{n}\big]\widehat{n}(\widehat{e_I\psi})\big\}\\
\notag&+\sum_{\iota_1\cup\iota_2\cup\iota_3=\iota,\,|\iota_3|<|\iota|}\partial^{\iota_1}n(\partial^{\iota_2}e_I^a)\partial_a\partial^{\iota_3}(\widehat{e_0\psi})-\sum_{\iota_1\cup\iota_2=\iota,\,|\iota_2|<|\iota|}\frac{1}{3}\partial^{\iota_1}\widehat{n}\partial^{\iota_2}(\widehat{e_I\psi}).
\end{align}
\begin{proposition}\label{prop:psi.est}
Let $N\ge4$ and assume the bootstrap assumptions \eqref{Boots} are valid.  Then the following estimate holds:
\begin{align}\label{high.psi.est}
\frac{1}{2}\partial_t[e^{2Ht}\|\widehat{e\psi}\|^2_{H^N(\Sigma_t)}]+2He^{2Ht}\|\widehat{e_0\psi}\|^2_{H^N(\Sigma_t)}\leq \eta \sum_{J=1}^3e^{3Ht}\|e_J\widehat{n}\|^2_{H^N(\Sigma_t)}+\frac{C}{\eta}e^{-\frac{1}{2}Ht}\mathcal{E}(t),
\end{align}
for all $t\in[0,T_{Boot})$, and some sufficiently small constant $\eta>0$ to be chosen later.
\end{proposition}
\begin{proof}
Employing the Sobolev inequality \eqref{Sob}, the bootstrap assumptions \eqref{Boots}, and the $C^{N-2}(\Sigma_t)$ bounds \eqref{CN-2.est}, we first deduce the error estimates:
\begin{align}\label{high.psi.est2}
\sum_{\mu=0}^3e^{Ht}\|\mathfrak{F}_\mu^\iota\|_{L^2(\Sigma_t)}\leq&\,Ce^{-\frac{3}{2}Ht}\sum_{J=1}^3e^{\frac{3}{2}Ht}\|e_J\widehat{n}\|_{H^N(\Sigma_t)}+Ce^{-Ht}\sqrt{\mathcal{E}(t)}
\end{align}
The standard energy identity for \eqref{high.e0psi.hat.eq}-\eqref{eIpsi.hat.eq} reads:
\begin{align}\label{high.psi.est3}
\notag&\frac{1}{2}\partial_t[e^{2Ht}\partial^\iota(\widehat{e_0\psi})\partial^\iota(\widehat{e_0\psi})]+\frac{1}{2}\partial_t[e^{2Ht}\partial^\iota(\widehat{e_I\psi})\partial^\iota(\widehat{e_I\psi})]+2H\partial^\iota(\widehat{e_0\psi})\partial^\iota(\widehat{e_0\psi})\\
=&\,e^{2Ht}ne_C[\partial^\iota(\widehat{e_0\psi})\partial^\iota(\widehat{e_C\psi})]+e^{2Ht}\partial^\iota(\widehat{e_0\psi})\mathfrak{F}_0^\iota+e^{2Ht}\partial^\iota(\widehat{e_I\psi})\mathfrak{F}_I^\iota\\
\notag&-\big[\mathcal{O}(e^{-2Ht})+\widehat{n}\big]\partial^\iota(\widehat{e_0\psi})\partial^\iota(\widehat{e_0\psi})
-\big[\mathcal{O}(e^{-2Ht})+\frac{1}{3}\widehat{n}\big]\partial^\iota(\widehat{e_I\psi})\partial^\iota(\widehat{e_I\psi})
\end{align}
Integrating \eqref{high.psi.est3} in $\Sigma_t$, summing over $|\iota|\leq N$, using the bound $\|\widehat{n}\|_{L^\infty(\Sigma_t)}\leq C\varepsilon e^{-\frac{3}{2}Ht}$, Young's inequality, and \eqref{high.psi.est2}, we arrive at the inequality:
\begin{align}\label{high.psi.est4}
\frac{1}{2}\partial_t[e^{2Ht}\|\widehat{e\psi}\|^2_{H^N(\Sigma_t)}]+2He^{2Ht}\|\widehat{e_0\psi}\|^2_{H^N(\Sigma_t)}\leq&\sum_{|\iota|\leq N}\int_{\Sigma_t}e^{2Ht}ne_C[\partial^\iota(\widehat{e_0\psi})\partial^\iota(\widehat{e_C\psi})]dx\\
\notag&+\eta \sum_{J=1}^3e^{3Ht}\|e_J\widehat{n}\|^2_{H^N(\Sigma_t)}+\frac{C}{\eta}e^{-\frac{1}{2}Ht}\mathcal{E}(t)
\end{align}
The first term in the last RHS is treated by integrating by parts with respect to $e_C$, using \eqref{IBP}:
\begin{align}\label{high.psi.est5}
\int_{\Sigma_t}e^{2Ht}ne_C[\partial^\iota(\widehat{e_0\psi})\partial^\iota(\widehat{e_C\psi})]dx=-\int_{\Sigma_t}e^{2Ht}\partial_a(ne_C^a)\partial^\iota(\widehat{e_0\psi})\partial^\iota(\widehat{e_C\psi})dx\leq Ce^{-Ht}\mathcal{E}(t).
\end{align}
Combining \eqref{high.psi.est4}-\eqref{high.psi.est5} we conclude \eqref{high.psi.est}.
\end{proof}
\subsection{Closing the bootstrap argument}

Combining the derived estimates in the Sections \ref{subsec:est.n}-\ref{subsec:est.psi}, we obtain an overall energy estimate for the total energy, which leads to an improvement of our bootstrap assumptions \eqref{Boots}. A standard continuation argument then yields the global existence of the perturbed solution in the future direction.
\begin{proposition}\label{prop:cont.arg}
Let $N\ge4$ and assume the bootstrap assumptions \eqref{Boots} are valid. Recall the definitions \eqref{thm.main.init}, \eqref{tot.en} of $\mathring{\varepsilon}$ and the total energy $\mathcal{E}(t)$. Then
 the following energy estimate holds:
\begin{align}\label{tot.en.est}
\partial_t\mathcal{E}(t)\leq C_Ne^{-\frac{1}{2}Ht}\mathcal{E}(t),
\end{align}
for all $t\in[0,T_{Boot})$. In particular, if $\mathring{\varepsilon}$ is sufficiently small such that $e^{2C_N}\mathring{\varepsilon}^2<\varepsilon^2$, \eqref{tot.en.est} yields a strict improvement of the bootstrap assumptions \eqref{Boots}. In  the latter case, $T_{Boot}=+\infty$ and the estimate
\begin{align}\label{tot.en.est2}
\mathcal{E}(t)\leq e^{2C_N}\mathring{\varepsilon}^2
\end{align}
holds for all $t\in[0,+\infty)$.
\end{proposition}
\begin{proof}
Adding the estimates in Propositions \ref{prop:n.est}, \ref{prop:eIi.est}, \ref{prop:k.gamma.est}, \ref{prop:psi.est}, and choosing $\eta$ sufficiently small, we infer \eqref{tot.en.est}. Integrating in $[0,t]$ then gives
\begin{align}\label{top.en.est3}
\mathcal{E}(t)\leq \exp\left\{\int^t_0C_Ne^{-\frac{\tau}{2}}d\tau\right\}\mathcal{E}(0)\leq e^{2C_N}\mathring{\varepsilon}^2,
\end{align}
for all $t\in[0,T_{Boot})$, which is an improvement of \eqref{Boots} provided $e^{2C_N}\mathring{\varepsilon}^2<\varepsilon^2$.

By standard local well-posedness (see Appendix \ref{sec:app}), if $\mathring{\varepsilon}$ is sufficiently small and $\widetilde{C}$ is sufficiently large, then there exists a
maximal time $T_{max} \in [0, +\infty)$, such that the solution $\widehat{k}_{IJ},\widehat{\gamma}_{IJB},\widehat{e}^i_I,\widehat{n},\widehat{e_0\psi},\widehat{e_I\psi}$ to the system \eqref{k.hat.eq}-\eqref{eIpsi.hat.eq} exists classically in $\{\Sigma_t\}_{t\in[0,T_{max})}$,
and such that the bootstrap assumptions \eqref{Boots} hold with $T_{Boot} = T_{max}$ and $\varepsilon^2:=\widetilde{C}\mathring{\varepsilon}^2$. By enlarging $\widetilde{C}$ if
necessary, we can assume that $\widetilde{C}>e^{2C_N}$.
Moreover, it is a standard result
that if $\varepsilon$ is sufficiently small, then either $T_{max}=+\infty$ or $T_{max} \in(0,+\infty)$ and
the bootstrap assumptions are saturated on the time interval $[0,T_{max})$, that is,
\begin{align}\label{satur}
\sup_{t\in[0,T_{max})}\mathcal{E}(t)=\varepsilon^2.
\end{align}
The latter possibility is ruled out by inequality \eqref{top.en.est3} when $\mathring{\varepsilon}$ is sufficiently small such that $e^{2C_N}\mathring{\varepsilon}^2<\varepsilon^2$. Thus, $T_{max}=+\infty$ and the estimate \eqref{top.en.est3} holds for all $t\in[0,+\infty)$.
\end{proof}
\section{Asymptotic behavior of solutions at infinity}\label{sec:asym.inf}

In this section we derive the precise asymptotic behavior of the perturbed solution, furnished by Proposition \ref{prop:cont.arg}, and identify its asymptotic data at $+\infty$. The necessary estimates to complete the proof of Theorem \ref{thm:main} are contained in Lemma \ref{lem:ref.est} and Proposition \ref{prop:metric.asym.inf}. The proof of Theorem \ref{thm:exp} is carried out in the end of the next subsection.

\subsection{Expansions for the reduced variables at $+\infty$}\label{subsec:red.exp}

Before computing the asymptotic expansions of the reduced variables, we must refine the estimate \eqref{tot.en.est2} for the variables $\widehat{k}_{IJ},\widehat{n},\widehat{e_0\psi}$.
Note that the stability estimate \eqref{tot.en.est2}, together with \eqref{Sob}, imply that
\begin{align}\label{CN-2.est2}
\begin{split}
\|\widehat{k}\|_{C^{N-2}(\Sigma_t)}
+\|\widehat{\gamma}\|_{C^{N-2}(\Sigma_t)}+\|\widehat{e}\|_{C^{N-2}(\Sigma_t)}+\|\widehat{e\psi}\|_{C^{N-2}(\Sigma_t)}\leq&\,C\mathring{\varepsilon} e^{-Ht},\\
\|\widehat{n}\|_{C^{N-2}(\Sigma_t)}\leq&\,C\mathring{\varepsilon} e^{-\frac{3}{2}Ht},
\end{split}
\end{align}
for all $t\in[0,+\infty)$.
\begin{lemma}\label{lem:ref.est}
Let $N\ge4$ and consider the solution furnished by Proposition \ref{prop:cont.arg}. The variables $\widehat{k}_{IJ},\widehat{n},\widehat{e_0\psi}$ satisfy the $C^{N-4}(\Sigma_t)$ bounds:
\begin{align}\label{k.n.e0psi.hat.ref.est}
\|\widehat{k}\|_{C^{N-4}(\Sigma_t)}+\|\widehat{n}\|_{C^{N-4}(\Sigma_t)}+\|\widehat{e_0\psi}\|_{C^{N-4}(\Sigma_t)}\leq C\mathring{\varepsilon}e^{-2Ht}.
\end{align}
for all $t\in[0,+\infty)$.
\end{lemma}
\begin{proof}
We only prove the assertion for $\widehat{k}$, the argument for $\widehat{n},\widehat{e_0\psi}$ is similar. Employing the $C^{N-2}(\Sigma_t)$ bounds \eqref{CN-2.est2}, we estimate the LHS of the evolution equation \eqref{k.hat.eq} for $\widehat{k}_{IJ}$. More precisely, we have the inequality:
\begin{align*}
\|\partial_t\widehat{k}_{IJ}+3H\widehat{k}_{IJ}\|_{C^{N-4}(\Sigma_t)}\leq C\mathring{\varepsilon}e^{-2Ht}.
\end{align*}
Using integrating factors, integrating in $[0,t]$, and using the triangle inequality for integrals, we obtain
\begin{align*}
\bigg\|\int^t_0\partial_\tau(e^{3Ht}\widehat{k}_{IJ})d\tau\bigg\|_{C^{N-4}(\Sigma_t)}\leq C\mathring{\varepsilon}e^{Ht}\qquad\Rightarrow\qquad\|\widehat{k}_{IJ}\|_{C^{N-4}(\Sigma_t)}\leq\|\widehat{k}_{IJ}\|_{C^{N-4}(\Sigma_0)}e^{-3Ht}+ C\mathring{\varepsilon}e^{-2Ht},
\end{align*}
as desired.
\end{proof}
With the refined estimates \eqref{k.n.e0psi.hat.ref.est} at our disposal, combined with \eqref{CN-2.est2}, we may now proceed to the

\begin{proof}[Proof of Theorem \ref{thm:exp}]

{\it Proof of \eqref{n.hat.exp}.} From the equation \eqref{n.hat.eq} of $\widehat{n}$, we infer that 
\begin{align}\label{n.hat.exp.est}
\|\partial_t\widehat{n}+2H\widehat{n}\|_{C^{N-6}(\Sigma_t)}\leq C\mathring{\varepsilon} e^{-4Ht}\qquad\Rightarrow\qquad\|\partial_t(e^{2Ht}\widehat{n})\|_{C^{N-6}(\Sigma_t)}\leq C\mathring{\varepsilon} e^{-2Ht}.
\end{align}
This implies that $e^{2Ht}\widehat{n}$ has a $C^{N-6}(\mathbb{T}^3)$ limit, $\widehat{n}^\infty(x)$, as $t\rightarrow+\infty$. Integrating \eqref{n.hat.exp.est} in $[t,+\infty)$, multiplying both sides with $e^{-2Ht}$, and using the triangle inequality for integrals we conclude \eqref{n.hat.exp}.

{\it Proof of \eqref{eIi.hat.exp}, \eqref{gamma.hat.exp}, \eqref{eIpsi.hat.exp}.} First, we use the equations \eqref{eIi.hat.eq}, \eqref{gamma.hat.eq}, \eqref{eIpsi.hat.eq}, and the estimates \eqref{CN-2.est2}, \eqref{k.n.e0psi.hat.ref.est} to deduce the inequalities
\begin{align*}
\notag\|\partial_t(e^{Ht}\widehat{e}_I^i)\|_{C^{N-4}(\Sigma_t)}\leq C\mathring{\varepsilon} e^{-2Ht},\\
\|\partial_t(e^{Ht}\widehat{\gamma}_{IJB})\|_{C^{N-5}(\Sigma_t)}\leq C\mathring{\varepsilon} e^{-2Ht},\\
\notag\|\partial_t(e^{Ht}\widehat{e_I\psi})\|_{C^{N-5}(\Sigma_t)}\leq C\mathring{\varepsilon} e^{-2Ht}.
\end{align*}
Then, we argue in the same manner as above. 

{\it Proof of \eqref{k.hat.exp}.} Recall the definition \eqref{eIi.infty} of $(e_I^a)^\infty(x)$ and the behavior of \eqref{red.var.FLRW} of $(e_I^a)_{FLRW}$ at $+\infty$. Plugging the already derived asymptotic behaviors for $\widehat{\gamma}_{IJB},\widehat{e_I\psi}$ into \eqref{k.hat.eq}, and using \eqref{CN-2.est2}, \eqref{k.n.e0psi.hat.ref.est}, we deduce the estimate:
\begin{align*}
&\big\|\partial_t\widehat{k}_{IJ}+3H\widehat{k}_{IJ}-\big\{(e_C^a)^\infty\partial_a\widehat{\gamma}_{IJC}^\infty-(e_I^a)^\infty\partial_a\widehat{\gamma}_{CJC}^\infty
-\frac{2}{3}\delta_{IJ}(e_C^a)^\infty\partial_a\widehat{\gamma}_{DDC}^\infty-\widehat{\gamma}_{CID}^\infty\widehat{\gamma}^\infty_{DJC}-\widehat{\gamma}_{IJD}^\infty\widehat{\gamma}_{CCD}^\infty\\
&+\frac{1}{3}\delta_{IJ}(\widehat{\gamma}_{CED}^\infty\widehat{\gamma}^\infty_{DEC}+\widehat{\gamma}_{EED}^\infty\widehat{\gamma}_{CCD}^\infty)-\widehat{e_I\psi}^\infty\widehat{e_J\psi}^\infty+\frac{1}{3}\delta_{IJ}\widehat{e_C\psi}^\infty\widehat{e_C\psi}^\infty\big\}(x)e^{-2Ht}\big\|_{C^{N-6}(\Sigma_t)}\leq C\mathring{\varepsilon} e^{-4Ht}
\end{align*}
Adopting the definition \eqref{infty.forcing.k.hat}, we rewrite
\begin{align}\label{k.hat.exp.est}
\notag\|\partial_t\widehat{k}_{IJ}+3H\widehat{k}_{IJ}-HF_{\widehat{k}}(x)e^{-2Ht}\|_{C^{N-6}(\Sigma_t)}\leq &\,C\mathring{\varepsilon} e^{-4Ht}\\
\notag\|\partial_t[\widehat{k}_{IJ}-F_{\widehat{k}}(x)e^{-2Ht}]+3H[\widehat{k}_{IJ}-F_{\widehat{k}}(x)e^{-2Ht}]\|_{C^{N-6}(\Sigma_t)}\leq &\,C\mathring{\varepsilon} e^{-4Ht}\\
\|\partial_t\big\{e^{3Ht}[\widehat{k}_{IJ}-F_{\widehat{k}}(x)e^{-2Ht}]\big\}\|_{C^{N-6}(\Sigma_t)}\leq &\,C\mathring{\varepsilon} e^{-Ht}
\end{align}
Hence, $e^{3Ht}[\widehat{k}_{IJ}-F_{\widehat{k}}(x)e^{-2Ht}]$ has a limit, $\widehat{k}^\infty_{IJ}(x)$, as $t\rightarrow+\infty$. The desired asymptotic behavior now follows by integrating \eqref{k.hat.exp.est} in $[t,+\infty)$, multiplying both sides with $e^{-3Ht}$, and using the triangle inequality for integrals.

{\it Proof of \eqref{e0psi.hat.exp}.} From the equation \eqref{e0psi.hat.eq}, the derived asymptotic behaviors of $\widehat{\gamma}_{IJB},\widehat{e_I\psi}$ and the estimates \eqref{CN-2.est2}, \eqref{k.n.e0psi.hat.ref.est}, we obtain:
\begin{align*}
\big\|\partial_t(\widehat{e_0\psi})+3H(\widehat{e_0\psi})-\big\{(e_C^a)^\infty(x)\partial_a(\widehat{e_C\psi})^\infty-\widehat{\gamma}_{CCD}^\infty(\widehat{e_D\psi})^\infty\big\}(x)\big\|_{C^{N-6}(\Sigma_t)}\leq C\mathring{\varepsilon}e^{-4Ht}.
\end{align*}
We then argue as in the proof of \eqref{k.hat.exp}, using definition \eqref{infty.forcing.e0psi.hat}.

{\it Proof of \eqref{hat.lim.est}.} Rewriting parts of the estimates \eqref{CN-2.est2}, \eqref{k.n.e0psi.hat.ref.est}
\begin{align*}
\|e^{2Ht}\widehat{n}\|_{C^{N-6}(\Sigma_t)},\|e^{Ht}\widehat{e}^i_I\|_{C^{N-4}(\Sigma_t)},\|e^{Ht}\widehat{\gamma}_{IJB}\|_{C^{N-5}(\Sigma_t)},\|e^{Ht}\widehat{e_I\psi}\|_{C^{N-5}(\Sigma_t)}\leq C\mathring{\varepsilon}
\end{align*}
implies the desired bound for $\widehat{n}^\infty(x),(\widehat{e}^i_I)^\infty(x),\widehat{\gamma}_{IJB}^\infty(x),(\widehat{e_I\psi})^\infty(x)$, by taking the limit $t\rightarrow+\infty$. These also imply the bounds 
\begin{align}\label{forcing.hat.est}
\|F_{\widehat{k}}(x)\|_{C^{N-6}(\mathbb{T}^3)},\|F_{\widehat{e_0\psi}}(x)\|_{C^{N-6}(\mathbb{T}^3)}\leq C\mathring{\varepsilon}
\end{align}
Integrating \eqref{k.hat.exp.est} in $[0,t]$, using the triangle inequality for integrals, and \eqref{forcing.hat.est}, we deduce the inequality:
\begin{align}\label{k.hat.exp.est2}
\|e^{3Ht}[\widehat{k}_{IJ}-F_{\widehat{k}}(x)e^{-2Ht}]\|_{C^{N-6}(\Sigma_t)}\leq &\,C\mathring{\varepsilon}+C\mathring{\varepsilon} e^{-Ht}.
\end{align}
Hence, taking the limit $t\rightarrow+\infty$ gives \eqref{hat.lim.est} for $\widehat{k}_{IJ}^\infty(x)$. 
The estimate for $(\widehat{e_0\psi})^\infty(x)$ is similar. This completes the proof of Theorem \ref{thm:exp}.
\end{proof}

\subsection{Expansions for the metric and scalar field at $+\infty$}

The asymptotic expansion of $\psi$ can be computed directly from that of $e_0\psi$. To compute the expansion of the spatial metric components, we introduce the new set of variables
\begin{align}\label{viC}
\partial_i=:v^C_ie_C.
\end{align}
According to \eqref{FLRW.sol}, \eqref{FLRW.asym}, the corresponding FLRW variables equal 
\begin{align}\label{viC.FLRW}
(v^C_i)_{FLRW}=\delta^C_ia(t)=\delta^C_i\mathring{a}\bigg(\frac{1}{2}\sqrt{\frac{\mathring{\varphi}_{FLRW}^2}{\Lambda}+1}+\frac{1}{2}\bigg)^{\frac{1}{3}}e^{Ht}+\delta^C_i\mathcal{O}(e^{-5Ht}).
\end{align}
The two definitions \eqref{ei}, \eqref{viC} imply the relations 
\begin{align}\label{viC.rel}
v^C_ie_C^j=\delta_i^j,\qquad e_C^iv^B_i=\delta_C^B. 
\end{align}
Using also \eqref{eIi.eq}, we obtain the evolution equation: 
\begin{align}\label{viC.eq}
e_0v^C_i=-k_{CB}v^B_i.
\end{align}
Setting $\widehat{v}^C_i:=v^C_i-(v^C_i)_{FLRW}$, we also have the analogous equation to \eqref{eIi.hat.eq}: 
\begin{align}\label{viC.hat.eq}
\partial_t\widehat{v}^C_i-\big[H+\mathcal{O}(e^{-2Ht})+\frac{1}{3}\widehat{n}\big]\widehat{v}^C_i=-n\widehat{k}_{CB}v^B_i+\big[H+\mathcal{O}(e^{-2Ht})+\frac{1}{3}\widehat{n}\big]\widehat{n}\widehat{v}^C_i+\mathcal{O}(e^{Ht})\widehat{n}+\mathcal{O}(e^{Ht})\widehat{n}^2.
\end{align}
Define the norms 
\begin{align}\label{vIC.norms}
\|\widehat{v}\|_{H^M(\Sigma_t)}=\sum_{i,C=1}^3\|\widehat{v}^C_i\|_{H^M(\Sigma_t)},\qquad\|\widehat{v}\|_{C^M(\Sigma_t)}=\sum_{i,C=1}^3\|\widehat{v}^C_i\|_{C^M(\Sigma_t)}.
\end{align}
Since $v_i^C$ can be solved in terms of $e_I^a$ using \eqref{viC.rel}, the initial assumption \eqref{thm.main.init} for $\widehat{e}^a_I$, together with the formula \eqref{red.var.FLRW} for $(e^a_I)_{FLRW}$ and the Sobolev inequality \eqref{Sob},  imply the initial bound:
\begin{align}\label{viC.init}
\|\widehat{v}\|_{H^N(\Sigma_0)}\leq C\mathring{\varepsilon}.
\end{align}
\begin{proposition}\label{prop:metric.asym.inf}
Let $N\ge4$ and consider the solution furnished by Proposition \ref{prop:cont.arg}. The variables $\widehat{v}^C_i$ satisfy the $C^{N-4}(\Sigma_t)$ bound:
\begin{align}\label{viC.hat.ref.est}
\|\widehat{v}\|_{C^{N-4}(\Sigma_t)}\leq C\mathring{\varepsilon}e^{Ht}.
\end{align}
Moreover, the variables $e^{-Ht}\widehat{v}^C_i,\widehat{\psi}$ have $C^{N-4}(\mathbb{T}^3)$ limits at infinity, $(\widehat{v}_i^C)^\infty(x), \widehat{\psi}^\infty(x)$, and the following estimates hold: 
\begin{align}\label{viC.hat.asym}
\begin{split}
\|\widehat{v}_i^C-(\widehat{v}_i^C)^\infty(x)e^{Ht}\|_{C^{N-4}(\Sigma_t)}\leq C\mathring{\varepsilon}e^{-Ht},\qquad\|\widehat{\psi}-\widehat{\psi}^\infty(x)\|_{C^{N-4}(\Sigma_t)}\leq C\mathring{\varepsilon} e^{-2Ht},\\
\|(\widehat{v}_i^C)^\infty(x)\|_{C^{N-4}(\mathbb{T}^3)}\leq C\mathring{\varepsilon},\qquad\|\widehat{\psi}^\infty(x)\|_{C^{N-4}(\mathbb{T}^3)}\leq C\mathring{\varepsilon}.
\end{split}
\end{align}
We also have the limits
\begin{align}\label{viC.infty}
\begin{split}
(v_i^C)^\infty(x):=&\,\lim_{t\rightarrow+\infty}e^{-Ht}v_i^C=(\widehat{v}_i^C)^\infty(x)+\delta^C_i\mathring{a}\bigg(\frac{1}{2}\sqrt{\frac{\mathring{\varphi}_{FLRW}^2}{\Lambda}+1}+\frac{1}{2}\bigg)^{\frac{1}{3}},\\
\psi^\infty(x):=&\,\lim_{t\rightarrow+\infty}\psi=\widehat{\psi}^\infty(x)+\psi^\infty_{FLRW}.
\end{split}
\end{align}

Furthermore, the spatial components of the spacetime metric \eqref{metric} that correspond to the global solution furnished by Proposition \ref{prop:cont.arg}, have the following asymptotic profile:
\begin{align}\label{metric.asym}
\|g_{ij}-e^{2Ht}g^\infty_{ij}(x)\|_{C^{N-4}(\Sigma_t)}\leq C,\qquad g^\infty_{ij}(x)=\delta_{BC}(v_i^B)^\infty(x)(v_j^C)^\infty(x)\in C^{N-4}(\mathbb{T}^3).
\end{align}
\end{proposition}
\begin{proof}
The estimate \eqref{viC.hat.ref.est} for $\widehat{v}_i^C$ can be easily bootstrapped using the bounds \eqref{k.n.e0psi.hat.ref.est} for $\widehat{k}_{IJ},\widehat{n}$, the equation \eqref{viC.hat.eq}, and the initial bound \eqref{viC.init}. From the equation \eqref{viC.hat.eq}, using \eqref{viC.hat.ref.est}, we then 
\begin{align}\label{viC.hat.asym2}
\|\partial_t(e^{-Ht}\widehat{v}^C_i)\|_{C^{N-4}(\Sigma_t)}\leq C\mathring{\varepsilon} e^{-2Ht}.
\end{align}
Hence, $e^{-Ht}\widehat{v}^C_i$ has a $C^{N-4}(\Sigma_t)$ limit $(\widehat{v}_i^C)^\infty(x)$, as $t\rightarrow+\infty$, and by integrating \eqref{viC.hat.asym2} in $[t,+\infty)$, we obtain
\begin{align}\label{viC.hat.asym3}
\|e^{-Ht}\widehat{v}^C_i-(\widehat{v}_i^C)^\infty(x)\|_{C^{N-4}(\Sigma_t)}\leq C\mathring{\varepsilon} e^{-2Ht},
\end{align}
which is the desired behavior for $\widehat{v}_i^C$. The third inequality in \eqref{viC.hat.asym} follows from \eqref{viC.hat.asym3} by taking the limit $t\rightarrow+\infty$ and using \eqref{viC.hat.ref.est}. The limit \eqref{viC.infty} for $v_i^C$ is immediate from the definition of $\widehat{v}_i^C$ and \eqref{viC.FLRW}. The assertion \eqref{metric.asym} follows from the formula
\begin{align}\label{gij.viC}
g_{ij}=g(\partial_i,\partial_j)=v_i^Bv_j^C\delta_{BC}
\end{align}
and the asymptotic behavior of $v_i^C=\widehat{v}_i^C+(v_i^C)_{FLRW}$.

For the difference of the scalar fields $\widehat{\psi}=\psi-\psi_{FLRW}$, we use the estimate \eqref{k.n.e0psi.hat.ref.est} for $\widehat{n},\widehat{e_0\psi}$. Recall the definition \eqref{diff.var} of $\widehat{e_0\psi}=e_0\psi-\partial_t\psi_{FLRW}$ to compute:
\begin{align}\label{psi.int}
\begin{split}
\|\partial_t\widehat{\psi}\|_{C^{N-4}(\Sigma_t)}=&\,\|\widehat{e_0\psi}+\partial_t\psi-e_0\psi\|_{C^{N-4}(\Sigma_t)}=\|\widehat{e_0\psi}+\widehat{n}e_0\psi\|_{C^{N-4}(\Sigma_t)}\\
\leq&\,\|\widehat{e_0\psi}\|_{C^{N-4}(\Sigma_t)}+\|\widehat{n}(\widehat{e_0\psi})\|_{C^{N-4}(\Sigma_t)}+\|\widehat{n}(\partial_t\psi_{FLRW})\|_{C^{N-4}(\Sigma_t)}\leq C\mathring{\varepsilon}e^{-2Ht}.
\end{split}
\end{align}
Hence, $\widehat{\psi}$ has a $C^{N-4}(\mathbb{T}^3)$ limit at $t=+\infty$, $\widehat{\psi}^\infty(x)$. Integrating \eqref{psi.int} in $[t,+\infty)$ and using the triangle inequality for integrals, we obtain the second bound in \eqref{viC.hat.asym}. Integrating \eqref{psi.int} in $[0,t]$, using the initial assumption \eqref{thm.main.init} for $\widehat{\psi}$ and the Sobolev inequality \eqref{Sob}, we obtain
\begin{align}\label{psi.est}
\|\widehat{\psi}\|_{C^{N-4}(\Sigma_t)}\leq C\mathring{\varepsilon}+C\mathring{\varepsilon}e^{-2Ht}.
\end{align}
Taking the limit $t\rightarrow+\infty$ gives the last bound in \eqref{viC.hat.asym}. 
The limit \eqref{viC.infty} for $\psi$ is immediate from the definition of $\widehat{\psi}=\psi-\psi_{FLRW}$ and the limit \eqref{FLRW.lim} for $\psi_{FLRW}$. 
This completes the proof of the proposition.
\end{proof}
\appendix
\section{Remarks on the well-posedness of the reduced system}\label{sec:app}

The evolution equations \eqref{k.eq}-\eqref{eIpsi.eq} do not have a well-posed initial value problem by themselves, because the equations \eqref{k.eq}-\eqref{gamma.eq} do not form a symmetric hyperbolic system for $k_{IJ},\gamma_{IJB}$. The problematic term responsible for this issue is $e_I\gamma_{CJC}$ in the RHS of \eqref{k.eq}. It is essentially the term that necessitates the use of the momentum constraint \eqref{momconst} in order to complete our energy argument, see Lemma \ref{lem:high.en.id} and Proposition \ref{prop:k.gamma.est}.  Nevertheless, the basic well-posedness properties required in the proof of Proposition \ref{prop:cont.arg} (local existence, continuation criteria etc.)  can be indirectly derived by considering a modified system. We describe two such arguments that have been successfully applied, in previous works, to variants of the present system, which are easily adaptable to our case.

\subsection{Well-posedness by symmetrization}

It is possible to symmetrize the equations \eqref{k.eq}-\eqref{gamma.eq} by adding appropriate multiples of the constraints. Let us describe this process by re-writing the equations \eqref{k.eq}-\eqref{gamma.eq} in the form
\begin{align}
\label{k.eq.app}e_0k_{IJ}+(n-1-\text{tr}k_{FLRW})k_{IJ}=&\, F_{IJ},\\
\label{gamma.eq.app}e_0\gamma_{IJB}-k_{IC} \gamma_{CJB}=&\,G_{IJB},
\end{align}
for short. We then introduce the modified equations:
\begin{align}
\label{k.eq.app2}e_0k_{IJ}+(n-1-\text{tr}k_{FLRW})k_{IJ}=&\, F_{IJ}+F_{JI},\\
\label{gamma.eq.app2}e_0\gamma_{IJB}-k_{IC} \gamma_{CJB}=&\,G_{IJB}-\delta_{IB}(e_C k_{CJ}+e_Jn-  
k_{JD}\gamma_{CCD}
-
k_{CD}\gamma_{CJD}+e_0\psi e_J\psi)\\
\notag&+\delta_{IJ}(e_C k_{CB}+e_Bn-  
k_{BD}\gamma_{CCD}
-
k_{CD}\gamma_{CBD}+e_0\psi e_B\psi)
\end{align}
The purpose of symmetrizing the RHS of \eqref{k.eq.app2} in $(I;J)$ is to automatically propagate the symmetry of $k_{IJ}$ off of the initial data in a local existence argument via a Picard iteration. The anti-symmetry of the RHS of \eqref{gamma.eq.app2} is manifest. As one can tediously check, the equations \eqref{k.eq.app2}-\eqref{gamma.eq.app2} constitute a symmetric hyperbolic system for $k_{IJ},\gamma_{IJB}$ (viewing the other variables as known coefficients or inhomogeneous terms). Hence, when coupled to the equations \eqref{eIi.eq}-\eqref{eIpsi.eq}, the local well-posedness of the initial value problem for the modified system, with initial data in $H^N(\Sigma_0)$, $N\ge4$, follows by standard theory,\footnote{By virtue of the Sobolev inequality \eqref{Sob}, the regularity class $H^N(\Sigma_t)$, $N\ge4$, implies $C^0$ control of up to two derivatives of the unknowns, which is more than enough for local existence of quasi-linear hyperbolic systems.} since the use of the momentum constraint \eqref{momconst} is no longer required to derive energy estimates.

Notice that if $k_{IJ},\gamma_{IJB}$ come from an actual solution to the Einstein vacuum equations, then \eqref{k.eq.app2}-\eqref{gamma.eq.app2} are equivalent to \eqref{k.eq.app}-\eqref{gamma.eq.app}. It therefore remains to show that the solution to the modified system gives a solution to \eqref{EE}-\eqref{divT}. This last step contains a fair amount of computations and it involves deriving a system of evolution equations for quantities that should vanish, namely, the components of the Ricci tensor that correspond to the constraints and the torsion of the connection induced by the solution $k_{IJ},\gamma_{IJB}$ to the modified equations \eqref{k.eq.app2}-\eqref{gamma.eq.app2}. For $n=1$, ie. a geodesic gauge, the complete argument can be found in \cite[Section 4]{FS2}. For a slightly different procedure applied to similar modified ADM-type systems with coordinate-based variables, see \cite{ST}.

\subsection{Well-posedness by studying a second order system}

An alternative argument involves the derivation of a second order equation for $k_{IJ}$, by taking the $e_0$ derivative of \eqref{k.eq} and using the other reduced equations \eqref{gamma.eq}-\eqref{eIpsi.eq} to replace the $e_0$ derivatives of the variables in the RHS. This results in a non-linear wave-type equation for $k_{IJ}$, coupled to the rest of the reduced equations, which now possess a locally well-posed initial value problem. As with the argument in the previous subsection, the Einstein vacuum equations have to be recovered in the end from the solution to the modified system. Such a procedure has been outlined in detail in the maximal gauge, see \cite[\S10.2]{CK}, ie. an elliptic equation for the lapse $n$ instead of a parabolic one, and for variables expressed with respect to a coordinate system instead of an orthonormal frame. However, the corresponding argument required for the setup in the present paper is similar. We also refer the reader to \cite[Theorem 14.1]{RS3} for a list of well-posedness properties for a similar formulation of Einstein's equations in a constant-mean-curvature gauge.

\end{document}